\documentclass[10pt]{article}
\usepackage[left=1.2in,top=1.5in,right=1.2in,bottom=2in,letterpaper]{geometry}

 \usepackage{ifpdf}
 
\usepackage[usenames]{color}

\usepackage[color,final]{showkeys}

\usepackage{amsmath,amssymb,amsthm}
\usepackage{latexsym,amsfonts,amscd,amsxtra,amstext}
\usepackage{algorithm,algorithmic}
\usepackage{url}

\usepackage{epsfig} 
\usepackage{threeparttable}
\usepackage{index}
\usepackage{subfigure}

\usepackage{xcolor}
\usepackage[bookmarks=false]{hyperref}
\usepackage{cite}

\hypersetup{
    bookmarks=false,         
    unicode=false,          
    pdftoolbar=true,        
    pdfmenubar=true,        
    pdffitwindow=false,     
    pdfstartview={FitH},    
    pdfnewwindow=true,      
    colorlinks=true,       
    linkcolor=red,          
    citecolor=green,        
    filecolor=magenta,      
    urlcolor=blue           
}

\usepackage[normalem]{ulem} 

\newcommand{\vb}{{\mathbf{b}}}

\newcommand{\vg}{{\mathbf{g}}}

\newcommand{\vk}{{\mathbf{k}}}

\newcommand{\vq}{{\mathbf{q}}}
\newcommand{\vr}{{\mathbf{r}}}
\newcommand{\vs}{{\mathbf{s}}}
\newcommand{\vt}{{\mathbf{t}}}
\newcommand{\vu}{{\mathbf{u}}}
\newcommand{\vv}{{\mathbf{v}}}
\newcommand{\vw}{{\mathbf{w}}}
\newcommand{\vx}{{\mathbf{x}}}
\newcommand{\vy}{{\mathbf{y}}}
\newcommand{\vz}{{\mathbf{z}}}

\newcommand{\vA}{{\mathbf{A}}}
\newcommand{\vB}{{\mathbf{B}}}
\newcommand{\vC}{{\mathbf{C}}}

\newcommand{\vI}{{\mathbf{I}}}

\newcommand{\vL}{{\mathbf{L}}}

\newcommand{\vT}{{\mathbf{T}}}

\newcommand{\cC}{{\mathcal{C}}}

\newcommand{\cG}{{\mathcal{G}}}
\newcommand{\cH}{{\mathcal{H}}}

\newcommand{\cP}{{\mathcal{P}}}

\newcommand{\vzero}{\mathbf{0}}

\newcommand{\dom}{{\mathrm{dom}}} 

\DeclareMathOperator*{\argmin}{arg\,min}

\newcommand{\bc}{\begin{center}}
\newcommand{\ec}{\end{center}}

\newcommand{\bdm}{\begin{displaymath}}
\newcommand{\edm}{\end{displaymath}}

\newcommand{\beq}{\begin{equation}}
\newcommand{\eeq}{\end{equation}}

\newcommand{\bfl}{\begin{flushleft}}
\newcommand{\efl}{\end{flushleft}}

\newcommand{\bt}{\begin{tabbing}}
\newcommand{\et}{\end{tabbing}}

\newcommand{\beqn}{\begin{eqnarray}}
\newcommand{\eeqn}{\end{eqnarray}}

\newcommand{\beqs}{\begin{align*}} 
\newcommand{\eeqs}{\end{align*}}  

\newcommand{\prox}{{\bf prox}}

\DeclareMathOperator*{\Min}{minimize}
\DeclareMathOperator*{\Max}{maximize}
\newcommand{\st}{\mathrm{subject~to}}

\newtheorem{theorem}{Theorem}

\newtheorem{assumption}{Assumption}
\newtheorem{definition}{Definition}
\newtheorem{corollary}{Corollary}
\newtheorem{remark}{Remark}
\newtheorem{lemma}{Lemma}
\newtheorem{proposition}{Proposition}

\usepackage{color}

\begin{document}

\title{Self Equivalence of the Alternating Direction Method of Multipliers}

\author{Ming Yan\thanks{Department of Mathematics, University of California, Los Angeles, CA 90095, USA. Emails: \texttt{yanm@math.ucla.edu} and \texttt{wotaoyin@math.ucla.edu} }\and Wotao Yin$^\ast$}
\date{\today}

\maketitle

\begin{abstract}
The alternating direction method of multipliers (ADM or ADMM) breaks a complex optimization problem into  much simpler subproblems. The ADM algorithms are typically short and easy to implement yet exhibit (nearly) state-of-the-art performance for large-scale optimization problems. 

To apply ADM, we first formulate a given problem into the ``ADM-ready" form, so the final algorithm depends on the formulation. A problem like $\Min_\vx u(\vx) + v(\vC\vx)$ has six different ``ADM-ready" formulations. They can be in the primal or dual forms, and they differ by how dummy variables are introduced. To each  ``ADM-ready" formulation, ADM can be applied in two different orders depending on how the primal variables are updated. Finally, we get twelve different ADM algorithms! How do they compare to each other? Which algorithm should one choose?  

In this chapter, we show that many of the  different ways of applying ADM are equivalent. Specifically, we show that ADM applied to a primal formulation is equivalent to ADM applied to  its Lagrange dual; ADM is equivalent to a primal-dual algorithm applied to the saddle-point formulation of the same problem. These results are surprising since the primal and dual variables in ADM are seemingly treated very differently, and some previous work exhibit preferences in one over the other on specific problems. In addition, when one of the two objective functions is quadratic, possibly subject to an affine constraint, we show that swapping the update order of the two primal variables in ADM gives the same algorithm. These results identify the few truly different ADM algorithms   for a problem,  which generally have different forms of subproblems from which it is easy to pick one with the most computationally friendly subproblems.
\end{abstract}


\section{Introduction}
The Alternating Direction Method of Multipliers (ADM or ADMM)  is a very popular algorithm with wide applications in signal and image processing, machine learning, statistics, compressive sensing, and operations
research. Combined with problem reformulation tricks, the method can reduce a complicated problem into much simpler subproblems.

The vanilla ADM applies to a linearly-constrained  problem with separable convex objective functions in the following ``ADM-ready" form:
\begin{align}\tag{P1}\label{for:P3-ADMM}
\left\{\begin{array}{ll}\Min\limits_{\vx,\vy} \quad& f(\vx)+g(\vy)\\
\st \quad& \vA\vx+\vB\vy=\vb,\end{array}\right.
\end{align}
where functions $f,g$ are proper, closed, convex  but not necessarily differentiable. ADM reduces~\eqref{for:P3-ADMM} into  two simpler subproblems  and  then iteratively updates $\vx$, $\vy$, as well as a multiplier (dual) variable $\vz$. 
Given  $(\vx^k,\vy^k,\vz^k)$, ADM generates  $(\vx^{k+1},\vy^{k+1},\vz^{k+1})$ as follows 
\begin{enumerate}
\item 
$\vy^{k+1} \in \argmin\limits_\vy g(\vy) +(2\lambda)^{-1}\|\vA\vx^k+\vB\vy-\vb+\lambda\vz^k\|_2^2,$\\[-15pt]
\item $\vx^{k+1} \in \argmin\limits_\vx f(\vx)+(2\lambda)^{-1}\|\vA\vx+\vB\vy^{k+1}-\vb+\lambda\vz^k\|_2^2,$\\[-15pt]
\item $\vz^{k+1} = \vz^k+\lambda^{-1}(\vA\vx^{k+1}+\vB\vy^{k+1}-\vb),$
\end{enumerate}
where $\lambda>0$ is a fixed parameter.
We use ``$\in$'' since the subproblems do not necessarily have unique solutions. 

Since $\{f,\vA,\vx\}$ and $\{g,\vB,\vy\}$ are in symmetric positions in \eqref{for:P3-ADMM}, swapping them does not change the problem. This corresponds to switching the order that $\vx$ and $\vy$ are updated in each iteration. But, since the variable updated first is used in the updating of the other variable, this swap leads to a different  sequence of  variables and thus a different algorithm. 

Note that the order switch does not change the per-iteration cost of ADM. Also note that one, however, cannot mix the two update orders at different iterations because it  will generally cause divergence, even when the primal-dual solution to \eqref{for:P3-ADMM} is unique.

\subsection{ADM works in many different ways}\label{sec:11}
In spite of  its popularity and vast literature,  there are still simple unanswered questions about ADM: how many ways can ADM be applied? and which ways work better? Before answering these questions, let us  examine  the following problem, to which we can find \textbf{twelve different ways to apply ADM:}  \beq\label{uv}\Min_{\vx}~ u(\vx) + v(\vC\vx),
\eeq
where  $u$ and $v$ are proper, closed, convex functions and $\vC$ is a linear mapping. Problem \eqref{uv} generalizes a large number of signal and image processing, inverse problem, and machine learning models. 

We shall reformulate 
\eqref{uv} into the form of \eqref{for:P3-ADMM}.  By introducing dummy variables in two different ways,  we obtain two  ADM-ready formulations of  problem \eqref{uv}: 
\begin{align}\label{uveq}
\left\{\begin{array}{ll}\Min\limits_{\vx,\vy} \quad& u(\vx)+v(\vy)\\
\st \quad& \vC\vx-\vy=0\end{array} \right.\quad\mbox{and}
&\qquad\left\{\begin{array}{ll}\Min\limits_{\vx,\bar{\vy}} \quad& u(\vx)+v(\vC\bar{\vy})\\
\st \quad& \vx-\bar{\vy}=0.\end{array} \right.
\end{align}
In addition, we can derive the dual problem of \eqref{uv}:
\beq\label{uvdual}
\Min_\vv~ u^*(-\vC^* \vv) + v^*(\vv),
\eeq
where $u^*,v^*$ are the convex conjugates (i.e., Legendre transforms) of functions $u,v$, respectively, $\vC^*$ is the adjoint of $\vC$, and $\vv$ is  the  dual variable. (The steps to derive \eqref{uvdual} from \eqref{uv} are standard and thus omitted.) Then, we also reformulate  \eqref{uvdual} into two ADM-ready forms, which use different dummy variables:
\begin{align}\label{uvdualeq}
\left\{\begin{array}{ll}\Min\limits_{\vu,\vv} \quad& u^*(\vu)+v^*(\vv)\\
\st \quad& \vu+\vC^* \vv=0\end{array} \right.\quad\mbox{and}
&\qquad\left\{\begin{array}{ll}\Min\limits_{\bar\vu,\vv} \quad& u^*(\vC^*\bar{\vu})+v^*(\vv)\\
\st \quad& \bar{\vu}+\vv=0.\end{array} \right.
\end{align}
Clearly, ADM can be applied to all of the four formulations in \eqref{uveq} and \eqref{uvdualeq}, and including the update order swaps, there are \emph{eight different ways} to apply ADM. 

Under some technical conditions such as the existence of saddle-point solutions, all the eight ADM will converge to a saddle-point solution or solutions for problem \eqref{uv}.   In short, they all work.

It is worth noting that by the Moreau identity, the subproblems involving $u^*$ and $v^*$ can be easily reduced to subproblems involving $u$ and $v$, respectively. No significant computing is required.

The two formulations in \eqref{uveq}, however, lead to significantly different ADM subproblems. In the ADM applied to the left formulation, $u$ and $\vC$ will appear in one subproblem and $v$ in the other subproblem. To the right formulation, $u$ will be alone while $v$ and $\vC$ will appear in the same subproblem. This difference applies to the two formulations in \eqref{uvdualeq} as well. It depends on the structures of $u,v,\vC$ to determine the better choices. Therefore, out of the eight, four will have   (more) difficult subproblems than the rest.

There are \emph{another four ways} to apply ADM to problem \eqref{uv}. Every one of them will have three  subproblems that separately involve $u,v,\vC$, so they are all different from the above eight.
 To get the first two, let us take the left formulation in \eqref{uveq} and introduce a dummy variable $\vs$,  obtaining  a new equivalent formulation
\beq\label{xyz} \left\{\begin{array}{ll}\Min\limits_{\vx,\vy,\vs} \quad& u(\vs)+v(\vy)\\
\st \quad& \vC\vx-\vy=0,\\
& \hspace{8pt}\vx-\vs\,=0. \end{array} \right.
\eeq
It turns out that the same ``dummy variable'' trick applied to the right formulation in \eqref{uveq} also gives  \eqref{xyz}, up to  a change of variable names. Although there are three variables, we can group $(\vy,\vs)$ and treat $\vx$ and $(\vy,\vs)$ as  the two variables. Then problem \eqref{xyz}  has the form \eqref{for:P3-ADMM}. Hence, we have two ways to apply ADM to  \eqref{xyz} with two different update orders. Note that $\vy$ and $\vs$ do not appear together in any equation or function, so  the ADM subproblem that updates $(\vy,\vs)$ will further decouple to two separable subproblems of $\vy$ and $\vs$; in other words, the resulting ADM has three subproblems   involving $\{\vx,\vC\}$ ,$\{\vy,v\}$, $\{\vs,u\}$ separately.  
The other two ways are results of the same ``dummy variable'' trick  applied to the either formulation in  \eqref{uvdualeq}. Again, since now $\vC$  has its own subproblem, these four ways are distinct from the previous eight ways.

As demonstrated through an example, there are quite many  ways to formulate the same optimization problem into ``ADM-ready" forms and obtain different ADM algorithms. 
While most ADM users  choose just one way without  paying much attention to the other choices,  some   show preferences toward  a specific formulation. For example, some prefer \eqref{xyz} over those in \eqref{uveq} and \eqref{uvdualeq} since $\vC$, $u$, $v$ all end up in separate subproblems. 
When applying ADM to certain $\ell_1$ minimization problems, the authors of~\cite{Yall1,yang2013dual} emphasize on the dual formulations, and later the authors of~\cite{Xiao2013Primal} show a preference over the primal formulations. When  ADM was proposed to solve a traffic equilibrium problem, it was first applied to the dual formulation in~\cite{Gabay1983} and, years later, to the primal formulation in~\cite{Fukushima1996The}. Regarding which one of the two variables should be updated first in ADM, neither a rule nor an equivalence claim is found in the literature. 
Other than giving preferences to ADM with simpler  subproblems, there is no results that compare the different formulations.

\subsection{Contributions}
This chapter shows that, applied to certain pairs of  different formulations of the same problem, ADM  will generate  equivalent sequences of variables that can be mapped exactly from one to another at every iteration. Specifically, between the sequence of an ADM algorithm on a primal formulation and that on the corresponding dual formulation, such maps exist. 

We also show that whenever at least one of $f$ and $g$ is a quadratic  function (including affine function as a special case), possibly subject to an affine constraint,  the sequence of an ADM algorithm can be mapped to that of the ADM algorithm using the opposite order for updating their variables. 

Abusing the word ``equivalence'', we say that ADM has  ``primal-dual equivalence'' and ``update-order equivalence (with a quadratic objective function).'' Equivalent ADM algorithms take the same number of iterations to reach the same accuracy. (However, it is possible that one algorithm is slightly better than the other in terms of numerical stability, for example, against round-off errors.) 

Equipped with these equivalence results, the first eight ways to apply ADM to problem~\eqref{uv} that were discussed in section \ref{sec:11} are reduced to four ways in light of primal-dual equivalence, and the four will  further reduce to two whenever $u$ or $v$, or both,  is a quadratic function. 

The last four ways to apply ADM on problem~\eqref{uv} discussed in section \ref{sec:11}, which yield three subproblems that separately involve $u$, $v$, and $\vC$, are all equivalent and reduce to just one due to primal-dual equivalence and one variable in them is associated with 0 objective  (for example, variable $\vx$ has 0 objective in problem \eqref{xyz}). 

Take the $\ell_p$-regularization problem,  $p\in [1,\infty]$, 
\beq\label{lassoprb}
\Min_{\vx}~ \|\vx\|_p + f(\vC\vx)
\eeq
as an example, which is special case of problem~\eqref{uv} with a quadratic function $u$ when $p=2$. We list its three  different  formulations, whose ADM algorithms are truly different,  as follows. When $p\not=2$ \emph{and} $f$ is non-quadratic, each of the first two formulations leads to a pair of different ADM algorithms with different orders of variable update; otherwise, each pair of algorithms is equivalent.
\begin{enumerate}
\item Left formulation of \eqref{uveq}: 
$$\left\{\begin{array}{ll}\Min\limits_{\vx,\vy} \quad& \|\vx\|_p+f(\vy)\\
\st \quad& \vC\vx-\vy=0.\end{array} \right.$$
The subproblem for $\vx$ involves $\ell_p$-norm and $\vC$. The other one for $\vy$ involves $f$.
\item Right formulation of \eqref{uveq}:
$$\left\{\begin{array}{ll}\Min\limits_{\vx,\vy} \quad& \|\vx\|_p+f(\vC\vy)\\
\st \quad& \vx-\vy=0.\end{array} \right.$$
The subproblem for $\vx$ involves $\ell_p$-norm and, for $p=1$ and $2$, has a closed-form solution. The other subproblem for $\vy$ involves $f(\vC\cdot)$.
\item Formulation \eqref{xyz}: for any $\mu>0$,
$$ \left\{\begin{array}{ll}\Min\limits_{\vx,\vy,\vs} \quad& \|\vs\|_p+f(\vy)\\
\st \quad& \vC\vx-\vy=0,\\
& \hspace{0pt}\mu(\vx-\vs)=0. \end{array} \right.$$
The subproblem for $\vx$ is quadratic program involving $\vC^*\vC+\mu\vI$. The subproblem for $\vs$ involves $\ell_p$-norm. The  subproblem for $\vy$ involves $f$. The subproblems for $\vs$ and $\vy$ are independent.
\end{enumerate}
The best choice depends on which has the simplest subproblems.

The result of ADM's primal-dual equivalence is surprising for three reasons. Firstly, ADM iteration updates \emph{two} primal variable, $\vx^k$ and $\vy^k$ in \eqref{for:P3-ADMM} and \emph{one} dual variable, all in different manners. The updates to the primal variables are done in a Gauss-Seidel manner and involve minimizing functions $f$ and $g$, but the update to the dual variable is explicit and linear. Surprisingly, ADM actually treats one of the two primal variables and the dual variable equally as we will later show. Secondly, most literature describes ADM as an inexact version of the Augmented Lagrangian Method (ALM) \cite{ALM}, which updates $(\vx,\vy)$ together rather than one after another. Although ALM maintains the primal variables,  under the hood ALM   is  the dual-only proximal-point algorithm that iterates  the dual variable. It is commonly believed that ADM is an inexact dual algorithm. Thirdly, primal and dual problems typically have different sizes and regularity properties, causing the same algorithm, even if it is applicable to both, to exhibit different performance. For example, the primal and dual variables may have different dimensions. If the primal function $f$ is Lipschitz differentiable, the dual function $f^*$ is strongly convex but can be non-differentiable, and vice versa. Such primal-dual differences often mean that it is numerically advantageous to solve one rather than the other, yet our result means that there is no such primal-dual difference on ADM.

Our maps between equivalent ADM sequences have very simple forms, as the reader will see below. Besides the technical proofs that establish the maps, it is interesting to mention the operator-theoretic perspective of our results. It is shown in \cite{Gabay1983} that the dual-variable sequence of ADM coincides with  a sequence of the Douglas-Rachford splitting (DRS) algorithm~\cite{Douglas1956on,Lions1979spliting}. Our ADM's primal-dual equivalence can be obtained through the above ADM--DRS relation and the Moreau identity: $\prox_h + \prox_{h^*}=\vI$, applied to the proximal maps of $f$ and $f^*$ and  those of $g$ and $g^*$. The details are omitted in this chapter. Here, $\prox_{h}(x) := \argmin_s h(s)+\frac{1}{2}\|s - x\|^2$. 


Our results of primal-dual equivalence for ADM extends to the Peaceman-Rachford splitting (PRS) algorithm. Let  the PRS operator~\cite{Peaceman1955the} be denoted as $\vT_{\mathrm{PRS}} = (2\prox_{f}-\vI)\circ(2\prox_{g}-\vI)$. The DRS operator is the average of the identity map and the PRS operator: $\vT_{\mathrm{DRS}}=\frac{1}{2}\vI+\frac{1}{2}\vT_{\mathrm{PRS}}$, and the Relaxed PRS (RPRS) operator is a weighted-average: $\vT_{\mathrm{RPRS}}=(1-\alpha)\vI+\alpha \vT_{\mathrm{PRS}}$, where $\alpha\in(0,1]$. The DRS and PRS algorithms that iteratively apply their operators to find a fixed point were originally proposed for evolving PDEs with two spatial dimensions in the 1950s and then extended to finding a root of the sum of two maximal monotone (set-valued) mappings by Lions and Mercier \cite{Lions1979spliting}. Eckstein showed, in~\cite[Chapter~3.5]{eckstein1989splitting}, that DRS/PRS applied to the primal problem~\eqref{uv} is equivalent to DRS/PRS applied to the dual problem~\eqref{uvdualeq} when $\vC=\vI$. We will show that RPRS applied to~\eqref{uv} is equivalent to RPRS applied to~\eqref{uvdual} for all $\vC$.



In addition to the aforementioned primal-dual and update-order equivalence, we obtain a primal-dual algorithm for the saddle-point formulation of \eqref{for:P3-ADMM} that is also equivalent to the ADM. This primal-dual algorithm is generally \emph{different} from the primal-dual algorithm proposed by Chambolle and Pock~\cite{chambolle2011first}, while they become the same in a special case. The connection between these two algorithms will be explained.

Even when using the same number of dummy variables, truly different ADM algorithms  can have different iteration complexities (do not confuse them with the difficulties of their subproblems). The convergence analysis of ADM, such as conditions for sublinear or linear convergence, involves many different scenarios \cite{DengYin2012a,DavisYin2014,DavisYin2014b}. The discussion of convergence rates of ADM algorithms is beyond the scope of this chapter. Our focus is on the equivalence.

\subsection{Organization} 
This chapter is organized as follows. Section~\ref{sec:notations} specifies our notation, definitions, and basic assumptions. The three equivalence results for ADM are shown in sections~\ref{sec:ADM},~\ref{sec:ADM-PD}, and~\ref{sec:ADM-equal2}: The primal-dual equivalence of ADM is discussed in sections~\ref{sec:ADM}; ADM is shown to be equivalent to a primal-dual algorithm applied to the saddle-point formulation in section~\ref{sec:ADM-PD}; In section~\ref{sec:ADM-equal2}, we show 
the update-order equivalence of ADM if $f$ or $g$ is a quadratic function, possibly subject to an affine constraint. The primal-dual equivalence of RPRS is shown in section~\ref{sec:DRS}. We conclude this chapter with  the application of our results on total variation image denoising in section~\ref{sec:app1}.

\section{Notation, definitions, and assumptions}\label{sec:notations}

Let $\cH_1$, $\cH_2$, and $\cG$ be (possibly infinite dimensional) Hilbert spaces. Bold lowercase letters such as $\vx$, $\vy$, $\vu$, and $\vv$ are used for points in the Hilbert spaces. In the example of \eqref{for:P3-ADMM}, we have $\vx\in\cH_1$, $\vy\in\cH_2$, and $\vb\in\cG$. When the Hilbert space a point belongs to is clear from the context, we do not specify it for the sake of simplicity. The inner product between points $\vx$ and $\vy$ is denoted by $\langle \vx,\vy\rangle$, and $\|\vx\|_2:=\sqrt{\langle \vx,\vx\rangle}$ is the corresponding norm. $\|\cdot\|_1$ and $\|\cdot\|_\infty$ denote the $\ell_1$ and $\ell_\infty$ norms, respectively. Bold uppercase letters such as $\vA$ and $\vB$ are used for both continuous linear mappings and matrices.  $\vA^*$ denotes the adjoint of $\vA$. $\vI$ denotes the identity mapping.

If $\cC$ is a convex and nonempty set, the indicator function $\iota_\cC$ is defined as follows:
\begin{align*}
\iota_\cC(\vx)  = \left\{\begin{array}{ll}0, & \mbox{ if }\vx\in \cC,\\ \infty, &\mbox{ if } \vx\notin \cC. \end{array}\right.
\end{align*}
Both lower and upper case letters such as $f$, $g$, $F$, and $G$ are used for functions. Let $\partial f(\vx)$ be the subdifferential of function $f$ at $\vx$. The proximal operator $\prox_{f(\cdot)}$ of function $f$ is defined as
\begin{align*}
\prox_{f(\cdot)}(\vx) = \argmin_\vy f(\vy) +\frac{1}{2}\|\vy-\vx\|_2^2,
\end{align*}
where the minimization  has the unique solution.
The convex conjugate $f^*$ of function $f$ is defined as
\begin{align*}
f^*(\vv) = \sup_\vx \{\langle \vv,\vx \rangle - f(\vx)\}.
\end{align*}
Let $\vL:\cH\rightarrow \cG$, the \emph{infimal postcomposition}~\cite[Def. 12.33]{bauschke2011convex} of $f:\cH\rightarrow(-\infty,+\infty]$ by $\vL$ is given by 
\begin{align*}
\vL\triangleright f: \vs\mapsto \inf f(\vL^{-1}(\vs)) = \inf_{\vx:\vL\vx=\vs}f(\vx),
\end{align*}
with $\dom(\vL\triangleright f)=\vL(\dom(f))$. 
\begin{lemma}\label{lemma:conjuga_infimal}If $f$ is convex and $\vL$ is affine and expressed as $\vL(\cdot)=\vA\cdot+\vb$, then $\vL\triangleright f$ is convex and the convex conjugate of $\vL\triangleright f$ can be found as follows:
\begin{align*} (\vL\triangleright f)^*(\cdot) = f^*(\vA^*\cdot)+\langle\cdot,\vb\rangle. \end{align*}
\end{lemma}
\begin{proof}Following from the definitions of convex conjugate and infimal postcomposition, we have 
\begin{align*}
(\vL\triangleright f)^*(\vv)&= \sup_{\vy}\langle \vv,\vy\rangle -\vL\triangleright f(\vy)=\sup_{\vx}\langle \vv,\vA\vx+\vb\rangle -f(\vx)\\
&=\sup_{\vx}\langle \vA^*\vv,\vx\rangle -f(\vx)+\langle \vv,\vb\rangle=f^*(\vA^*\vv)+\langle \vv,\vb\rangle.
\end{align*}
\end{proof}


\begin{definition}An algorithm on one problem is \emph{equivalent to} another algorithm on the same or another equivalent problem means that the steps in one algorithm can be recovered from the steps in another algorithm, with proper initial conditions and parameters.
\end{definition}

\begin{definition}An optimization algorithm is called \emph{primal-dual equivalent} if this algorithm applied to the primal formulation is equivalent to the same algorithm applied to its Lagrange dual.
\end{definition}
It is important to note that most algorithms are not primal-dual equivalent. ALM applied to the primal problem is equivalent to proximal point method applied to the dual problem~\cite{rockafellar1973dual}, but both algorithms are not primal-dual equivalent. In this chapter, we will show that ADM and RPRS are primal-dual equivalent.

We make the following assumptions throughout the chapter:
\begin{assumption}Functions in this chapter are assumed to be proper, closed, and convex. 
\end{assumption}
\begin{assumption}The saddle-point solutions to all the optimization problems in this chapter are assumed to exist.
\end{assumption}

\section{Equivalent problems}\label{sec:problems}

A \emph{primal formulation} equivalent to ~\eqref{for:P3-ADMM} is
\begin{align}\tag{P2}\label{for:P4-ADMM}
\left\{\begin{array}{ll}
\Min\limits_{\vs,\vt} \quad& F(\vs)+G(\vt)\\
\st \quad& \vs+\vt=\vzero,\end{array}\right.
\end{align}
where $\vs,\vt\in \cG$ and
\begin{subequations}\label{for:subprobs}
\begin{align}
F(\vs) &:= \min\limits_{\vx}f(\vx)+\iota_{\{\vx:\vA\vx=\vs\}}(\vx),\\
G(\vt) &:=\min\limits_\vy g(\vy)+\iota_{\{\vy:\vB\vy-\vb=\vt\}}(\vy).\label{for:subprobG}
\end{align}
\end{subequations}

\begin{remark}If we define $\vL_f$ and $\vL_g$ as $\vL_f(\vx)=\vA\vx$ and $\vL_g(\vy)=\vB\vy-\vb$, respectively, then 
$$F=\vL_f\triangleright f,\qquad G=\vL_g\triangleright g.$$
\end{remark}

The Lagrange dual of~\eqref{for:P3-ADMM}  is
\begin{align}\label{for:dualtoP1}
\Min_\vv  \quad f^*(-\vA^*\vv) + g^*(-\vB^*\vv)+\langle\vv,\vb\rangle,
\end{align}
which can be derived from $\Min\limits_\vv\left(-\min\limits_{\vx,\vy}L(\vx,\vy,\vv)\right)$ with the Lagrangian defined as follows:
\begin{align*}
L(\vx,\vy,\vv) = f(\vx)+g(\vy) + \langle \vv,\vA\vx+\vB\vy-\vb \rangle.
\end{align*}
An \emph{ADM-ready} formulation of ~\eqref{for:dualtoP1} is\begin{align}\tag{D1}\label{for:D3-ADMM}
\left\{\begin{array}{ll}
\Min\limits_{\vu,\vv} \quad& f^*(-\vA^*\vu)+g^*(-\vB^*\vv)+\langle \vv,\vb\rangle\\
\st \quad& \vu-\vv=\vzero.\end{array}\right.
\end{align}
When ADM is applied to an ADM-ready formulation of the Lagrange dual problem, we call it \emph{Dual ADM}. The original ADM is called \emph{Primal ADM}.

Following similar steps,  the ADM ready formulation of the Lagrange dual of~\eqref{for:P4-ADMM} is
\begin{align}\tag{D2}\label{for:D4-ADMM}
\left\{\begin{array}{ll}
\Min\limits_{\vu,\vv} \quad& F^*(-\vu)+G^*(-\vv)\\
\st \quad& \vu-\vv=\vzero.\end{array}\right.
\end{align}
The equivalence between~\eqref{for:D3-ADMM} and~\eqref{for:D4-ADMM} is trivial since
\begin{align*}
F^*(\vu)&=f^*(\vA^*\vu),\\
G^*(\vv)&=g^*(\vB^*\vv)-\langle \vv,\vb\rangle,
\end{align*}
which follows from Lemma~\ref{lemma:conjuga_infimal}.

Although there can be multiple equivalent formulations of the same problem (e.g.,~\eqref{for:P3-ADMM},~\eqref{for:P4-ADMM},~\eqref{for:dualtoP1}, and ~\eqref{for:D3-ADMM}/\eqref{for:D4-ADMM} are equivalent), an algorithm may or may not be applicable to some of them. Even when they are, on different formulations, their behaviors such as convergence and speed of convergence are different. In particular, most algorithms have different behaviors on primal and dual formulations of the same problem. An algorithm applied to a primal formulation does not dictate the behavior of the same algorithm applied to the related dual formulation. The simplex method in linear programming has different performance when applied to both the primal and dual problems, i.e., the primal simplex method starts with a primal basic feasible solution (dual infeasible) until the dual feasibility conditions are satisfied, while the dual simplex method starts with a dual basic feasible solution (primal infeasible) until the primal feasibility conditions are satisfied. The ALM also has different performance when applied to the primal and dual problems, i.e., ALM applied to the primal problem is equivalent to proximal point method applied to the related dual problem, and proximal point method is, in general, different from ALM on the same problem.

\section{Primal-dual equivalence of ADM}\label{sec:ADM}
In this section we show the primal-dual equivalence      of ADM. Algorithms~\ref{alg:ADM-1}-\ref{alg:ADM-3}~describe how ADM is applied to~\eqref{for:P3-ADMM},~\eqref{for:P4-ADMM}, and~\eqref{for:D3-ADMM}/~\eqref{for:D4-ADMM}\cite{ADM1,ADM2}.

\begin{algorithm}[H]
\caption{ADM on~\eqref{for:P3-ADMM}}
\label{alg:ADM-1}
\begin{algorithmic}
\STATE initialize $\vx_1^0$, $\vz_1^0$, $\lambda>0$
\FOR {$k =0,1, \cdots$}
\STATE $\vy_1^{k+1} \in \argmin\limits_\vy g(\vy) +(2\lambda)^{-1}\|\vA\vx_1^k+\vB\vy-\vb+\lambda\vz_1^k\|_2^2$
\STATE $\vx_1^{k+1} \in \argmin\limits_\vx f(\vx)+(2\lambda)^{-1}\|\vA\vx+\vB\vy_1^{k+1}-\vb+\lambda\vz_1^k\|_2^2$
\STATE $\vz_1^{k+1} = \vz_1^k+\lambda^{-1}(\vA\vx_1^{k+1}+\vB\vy_1^{k+1}-\vb)$
\ENDFOR
\end{algorithmic}
\end{algorithm}
\begin{algorithm}[H]
\caption{ADM on~\eqref{for:P4-ADMM}}
\label{alg:ADM-2}
\begin{algorithmic}
\STATE initialize $\vs_2^0$, $\vz_2^0$, $\lambda>0$
\FOR {$k =0,1, \cdots$}
\STATE $\vt_2^{k+1} = \argmin\limits_\vt G(\vt)+(2\lambda)^{-1}\|\vs_2^k+\vt+\lambda\vz_2^k\|_2^2$
\STATE $\vs_2^{k+1} = \argmin\limits_\vs F(\vs)+(2\lambda)^{-1}\|\vs+\vt_2^{\vk+1}+\lambda\vz_2^k\|_2^2$
\STATE $\vz_2^{k+1} = \vz_2^k+\lambda^{-1}(\vs_2^{k+1}+\vt_2^{k+1})$
\ENDFOR
\end{algorithmic}
\end{algorithm}
\begin{algorithm}[H]
\caption{ADM on~\eqref{for:D3-ADMM}/\eqref{for:D4-ADMM}}
\label{alg:ADM-3}
\begin{algorithmic}
\STATE initialize $\vu_3^0$, $\vz_3^0$, $\lambda>0$
\FOR {$k =0,1, \cdots$}
\STATE $\vv_3^{k+1} = \argmin\limits_\vv G^*(-\vv)+{\lambda\over 2}\|\vu_3^k-\vv+\lambda^{-1}\vz_3^k\|_2^2$
\STATE $\vu_3^{k+1} = \argmin\limits_\vu F^*(-\vu)+{\lambda\over2}\|\vu-\vv_3^{\vk+1}+\lambda^{-1}\vz_3^k\|_2^2$
\STATE $\vz_3^{k+1} = \vz_3^k+\lambda(\vu_3^{k+1}-\vv_3^{k+1})$
\ENDFOR
\end{algorithmic}
\end{algorithm}

The $\vy_1^{k}$ and $\vx_1^{k}$ in Algorithm~\ref{alg:ADM-1} may not be unique because of the matrices $\vA$ and $\vB$, while $\vA\vx^k_1$ and $\vB\vy^k_1$ are unique. In addition, $\vA\vx^k_1$ and $\vB\vy^k_1$ are calculated for twice and thus stored in the implementation of Algorithm~\ref{alg:ADM-1} to save the second calculation. Following the equivalence of Algorithms~\ref{alg:ADM-1} and~\ref{alg:ADM-2} in Part 1 of the following theorem~\ref{thm:ADMM-equal}, we can view problem \eqref{for:P4-ADMM}  as the \emph{master problem} of \eqref{for:P3-ADMM}.  We can say that ADM is essentially an algorithm applied only to the master problem \eqref{for:P4-ADMM}, which is Algorithm ~~\ref{alg:ADM-2}; this fact has been obscured by the often-seen Algorithm ~\ref{alg:ADM-1}, which  integrates ADM on the master problem with the independent subproblems in~\eqref{for:subprobs}.

\begin{theorem}[Equivalence of Algorithms~\ref{alg:ADM-1}-\ref{alg:ADM-3}] \label{thm:ADMM-equal}
Suppose $\vA\vx_1^0=\vs_2^0=\vz_3^0$ and $\vz_1^0=\vz_2^0=\vu_3^0$ and that the same parameter $\lambda$ is used in Algorithms~\ref{alg:ADM-1}-\ref{alg:ADM-3}. Then, their equivalence can be established as follows:
\begin{enumerate}
\item 
From $\vx_1^k$, $\vy_1^k$, $\vz_1^k$ of Algorithm~\ref{alg:ADM-1}, we obtain $\vt_2^k$, $\vs_2^k$, $\vz_2^k$ of Algorithm~\ref{alg:ADM-2} through:
\begin{subequations}
\begin{align}
\vt_2^k&=\vB\vy_1^k-\vb,\label{p1vt}\\
\vs_2^k&=\vA\vx_1^k,\label{p1vs}\\
\vz_2^k&=\vz_1^k.\label{p1vz}
\end{align}
\end{subequations}
From $\vt_2^k$, $\vs_2^k$, $\vz_2^k$ of Algorithm~\ref{alg:ADM-2}, we obtain $\vy_1^k$, $\vx_1^k$, $\vz_1^k$ of Algorithm~\ref{alg:ADM-1} through:
\begin{subequations}
\begin{align}
\vy_1^k &= \argmin_\vy\{g(\vy):  \vB\vy-\vb=\vt_2^k\},\label{p2vy}\\
\vx_1^k &= \argmin_\vx\{f(\vx): \vA\vx=\vs_2^k\},\label{p2vx}\\
\vz_1^k&=\vz_2^k.\label{p2vz}
\end{align}
\end{subequations}
\item 
We can recover the iterates of Algorithm~\ref{alg:ADM-2} and ~\ref{alg:ADM-3} from  each other through
\begin{align}\label{p3uz}
 \vu_3^k=\vz_2^k,\qquad \vz_3^k=\vs_2^k.
\end{align}
\end{enumerate}
\end{theorem}

\begin{proof}
\emph{Part 1.} Proof by induction.\\
We argue that under \eqref{p1vs} and \eqref{p1vz}, Algorithms 1 and 2 have \emph{essentially identical} subproblems in their \emph{first} steps at the $k$th iteration.  Consider the following problem, which is obtained by plugging the definition of $G(\cdot)$  into the $\vt_2^{k+1}$-subproblem of  Algorithm~\ref{alg:ADM-2}:
\begin{align}\label{for:ytsub}
(\vy_1^{k+1},\vt_2^{k+1})=\argmin_{\vy,\vt} g(\vy)+\iota_{\{(\vy,\vt):\vB\vy-\vb=\vt\}}(\vy,\vt)+(2\lambda)^{-1}\|\vs_2^k+\vt+\lambda\vz_2^k\|_2^2.
\end{align}
If one minimizes over $\vy$ first while keeping $\vt$ as a variable, one eliminates $\vy$ and recovers the $\vt_2^{k+1}$-subproblem of  Algorithm~\ref{alg:ADM-2}. If one minimizes over $\vt$ first while keeping $\vy$ as a variable, then after plugging in \eqref{p1vs}  and \eqref{p1vz}, problem \eqref{for:ytsub} reduces to the  $\vy_1^{k+1}$-subproblem of  Algorithm~\ref{alg:ADM-1}.
In addition, $(\vy_1^{k+1},\vt_2^{k+1})$  obeys
\begin{align}\vt_2^{k+1}=\vB\vy_1^{k+1}-\vb, \label{for:tbBy}
\end{align}
which is \eqref{p1vt} at ${k+1}$. Plugging $\vt = \vt_2^{k+1}$ into \eqref{for:ytsub} yields problem \eqref{p2vy} for $\vy_1^{k+1}$, which must be equivalent to the  $\vy_1^{k+1}$-subproblem of  Algorithm~\ref{alg:ADM-2}. Therefore,  the  $\vy_1^{k+1}$-subproblem of  Algorithm~\ref{alg:ADM-1} and the $\vt_2^{k+1}$-subproblem of  Algorithm~\ref{alg:ADM-2} are equivalent through \eqref{p1vt} and  \eqref{p2vy} at $k+1$, respectively.

Similarly, under \eqref{for:tbBy} and \eqref{p1vz}, we can show that the  $\vx_1^{k+1}$-subproblem of  Algorithm~\ref{alg:ADM-1} and the $\vs_2^{k+1}$-subproblem of  Algorithm~\ref{alg:ADM-2} are equivalent through the formulas for \eqref{p1vs} and  \eqref{p2vx} at $k+1$, respectively.

Finally, under \eqref{p1vt} and \eqref{p1vs} at $k+1$ and $\vz_2^k=\vz_1^k$, the formulas for $\vz_1^{k+1}$ and $\vz_2^{k+1}$ in Algorithms 1 and 2 are identical, and they return $\vz_1^{k+1}=\vz_2^{k+1}$, which is \eqref{p1vz} and  \eqref{p2vz} at $k+1$.

\emph{Part 2.}
Proof by induction. Suppose that \eqref{p3uz} holds. We shall show that \eqref{p3uz} holds at $k+1$. Starting from the optimality condition of the $\vt_2^{k+1}$-subproblem of Algorithm~\ref{alg:ADM-2}, we derive
\begin{align*}
& \vzero\in \partial G(\vt_2^{k+1})+\lambda^{-1} (\vs_2^k+\vt_2^{k+1}+\lambda\vz_2^k)\\
\Longleftrightarrow~ & \vt_2^{k+1}\in\partial G^*(-\lambda^{-1} (\vs_2^k+\vt_2^{k+1}+\lambda\vz_2^k))\\
\Longleftrightarrow~ & \lambda\left[\lambda^{-1} (\vs_2^k+\vt_2^{k+1}+\lambda\vz_2^k)\right]-(\lambda\vz_2^k+\vs_2^k)\in\partial G^*(-\lambda^{-1} (\vs_2^k+\vt_2^{k+1}+\lambda\vz_2^k))\\
\Longleftrightarrow~ & -\lambda\left[\lambda^{-1} (\vs_2^k+\vt_2^{k+1}+\lambda\vz_2^k)\right]+(\lambda\vu_3^k+\vz_3^k)\in-\partial G^*(-\lambda^{-1} (\vs_2^k+\vt_2^{k+1}+\lambda\vz_2^k))\\
\Longleftrightarrow~ & \vzero\in -\partial G^*(-\lambda^{-1} (\vs_2^k+\vt_2^{k+1}+\lambda\vz_2^k))-\lambda\left[\vu_3^k-\lambda^{-1} (\vs_2^k+\vt_2^{k+1}+\lambda\vz_2^k)+\lambda^{-1}\vz_3^k\right]\\
\Longleftrightarrow~ & \vv_3^{k+1}=\lambda^{-1} (\vs_2^k+\vt_2^{k+1}+\lambda\vz_2^k)=\lambda^{-1} (\vz_3^k+\vt_2^{k+1}+\lambda\vz_2^k),
\end{align*}
where the last equivalence follows from the  optimality condition for the $\vv_3^{k+1}$-subproblem of Algorithm~\ref{alg:ADM-3}.

Starting from the optimality condition of the $\vs_2^{k+1}$-subproblem of Algorithm~\ref{alg:ADM-2}, and applying the update, $\vz_2^{k+1}=\vz_2^k+\lambda^{-1}(\vs_2^{k+1}+\vt_2^{k+1})$, in Algorithm~\ref{alg:ADM-2} and the identity of $\vt_2^{k+1}$ obtained above,  we derive \begin{align*}
&\vzero\in \partial F(\vs_2^{k+1})+\lambda^{-1}(\vs_2^{k+1}+\vt_2^{k+1}+\lambda\vz_2^k)\\
\Longleftrightarrow~ &\vzero\in \partial F(\vs_2^{k+1})+\vz_2^{k+1}\\
\Longleftrightarrow~ &\vzero\in \vs_2^{k+1}-\partial F^*(-\vz_2^{k+1})\\
\Longleftrightarrow~ &\vzero\in \lambda(\vz_2^{k+1}-\vz_2^k)-\vt_2^{k+1}-\partial F^*(-\vz_2^{k+1})\\
\Longleftrightarrow~ &\vzero\in \lambda(\vz_2^{k+1}-\vz_2^k)+\vz_3^k+\lambda (\vz_2^k-\vv_3^{k+1})-\partial F^*(-\vz_2^{k+1})\\
\Longleftrightarrow~ &\vzero\in -\partial F^*(-\vz_2^{k+1})+\lambda(\vz_2^{k+1}-\vv_3^{k+1}+\lambda^{-1}\vz_3^k)\\
\Longleftrightarrow~ & \vz_2^{k+1} = \vu_3^{k+1}.
\end{align*}
where the last equivalence follows from the optimality condition for the $\vu_3^{k+1}$-subproblem of Algorithm~\ref{alg:ADM-3}. Finally, combining the update formulas of $\vz_2^{k+1}$ and $\vz_3^{k+1}$ in Algorithm~~\ref{alg:ADM-2} and \ref{alg:ADM-3}, respectively, as well as the identities for $\vu_3^{k+1}$ and $\vv_3^{k+1}$ obtained above, we obtain
\begin{align*}
\vz_3^{k+1}&=\vz_3^{k}+\lambda(\vu_3^{k+1}-\vv_3^{k+1})=\vs^k+\lambda (\vz_2^{k+1}-\vz_2^k-\lambda^{-1}(\vs_2^k+\vt_2^{k+1}))\\
&=\lambda (\vz_2^{k+1}-\vz_2^k)-\vt_2^{k+1} = \vs_2^{k+1}.
\end{align*}
\end{proof}

\begin{remark} Part 2 of the theorem (ADM's primal-dual equivalence) can also be derived by combining the following two equivalence results:\ (i) the equivalence between ADM on the primal problem and the Douglas-Rachford splitting (DRS) algorithm~\cite{Douglas1956on,Lions1979spliting} on the dual problem~\cite{Gabay1983}, and (ii) the equivalence result between DRS algorithms applied to the master problem~\eqref{for:P4-ADMM} and its dual problem (cf.~\cite[Chapter~3.5]{eckstein1989splitting}\cite{eckstein1994some}). In this chapter, however, we  provide an elementary  algebraic proof in order to derive  the  formulas in theorem~\ref{thm:ADMM-equal} that recover the iterates of one algorithm from another.
\end{remark}

Part 2 of the theorem shows that ADM is a symmetric primal-dual  algorithm. The reciprocal positions of parameter $\lambda$ indicates its function to ``balance'' the primal and dual progresses.

Part 2 of the theorem also shows that Algorithms~\ref{alg:ADM-2} and~\ref{alg:ADM-3} have no difference, in terms of per-iteration complexity and the number of iterations needed to reach an accuracy. However, Algorithms~\ref{alg:ADM-1} and~\ref{alg:ADM-2} have difference in terms of per-iteration complexity. In fact, Algorithm~\ref{alg:ADM-2} is implemented for Algorithm~\ref{alg:ADM-1} because Algorithm~\ref{alg:ADM-2} has smaller complexity than Algorithm~\ref{alg:ADM-1}. See the examples in sections~\ref{sec:examples41} and~\ref{sec:examples42}.

\subsection{Primal-dual equivalence of ADM on~\eqref{uv} with three subproblems}
In section~\ref{sec:11}, we introduced four different ways to apply ADM on~\eqref{uv} with three subproblems. The ADM-ready formulation for the primal problem is~\eqref{xyz}, and the ADM applied to this formulation is 
\begin{subequations}\label{alg:xyz_primal}
\begin{align}
\vx^{k+1} &= \argmin\limits_\vx \|\vx-\vs^k+\lambda\vz_\vs^k\|_2^2+\|\vC\vx-\vy^k+\lambda\vz_\vy^k\|_2^2,\\
\vs^{k+1} &= \argmin\limits_\vs u(\vs) +(2\lambda)^{-1}\|\vx^{k+1}-\vs+\lambda\vz_\vs^k\|_2^2,\\
\vy^{k+1} &= \argmin\limits_\vy v(\vy)+(2\lambda)^{-1}\|\vC\vx^{k+1}-\vy+\lambda\vz_\vy^k\|_2^2,\\
\vz_\vs^{k+1} &= \vz_\vs^k+\lambda^{-1}(\vx^{k+1}-\vs^{k+1}),\\
\vz_\vy^{k+1} &= \vz_\vy^k+\lambda^{-1}(\vC\vx^{k+1}-\vy^{k+1}).
\end{align}
\end{subequations}
Similarly, we can introduce a dummy variable $\vt$ into the left formulation in~\eqref{uvdualeq} and obtain a new equivalent formulation
\beq\label{xyz2} 
\left\{\begin{array}{ll}\Min\limits_{\vu,\vv,\vt} \quad& u^*(\vu)+v^*(\vt)\\
\st \quad& \vC^*\vv+\vu=0, ~\vv-\vt=0. \end{array} \right.
\eeq
The ADM applied to~\eqref{xyz2} is
\begin{subequations}\label{alg:xyz_dual}
\begin{align}
\vv^{k+1} &= \argmin\limits_\vv \|\vC^*\vv+\vu^k+\lambda^{-1}\vz_\vu^k\|_2^2+\|\vv-\vt^k+\lambda^{-1}\vz_\vt^k\|_2^2,\\
\vu^{k+1} &= \argmin\limits_\vu u^*(\vu) +{\lambda\over2}\|\vC^*\vv^{k+1}+\vu+\lambda^{-1}\vz_\vu^k\|_2^2,\\
\vt^{k+1} &= \argmin\limits_\vt v^*(\vt) +{\lambda\over2}\|\vv^{k+1}-\vt+\lambda^{-1}\vz_\vt^k\|_2^2,\\
\vz_\vu^{k+1} &= \vz_\vu^k+\lambda(\vC^*\vv^{k+1}+\vu^{k+1}),\\
\vz_\vt^{k+1} &= \vz_\vt^k+\lambda(\vv^{k+1}-\vt^{k+1}).
\end{align}
\end{subequations}

Interestingly, as shown in the following corollary, ADM algorithms~\eqref{alg:xyz_primal} and~\eqref{alg:xyz_dual} applied to~\eqref{xyz} and ~\eqref{xyz2} are equivalent. 
\begin{corollary}
If the initialization for algorithms~\eqref{alg:xyz_primal} and~\eqref{alg:xyz_dual} satisfies $\vz_\vy^0 =  \vt^0$, $\vz_\vs^0=\vu^0$, $\vs^0=-\vz_\vu^0$, and $\vy^0=\vz_\vt^0$. Then for $k\geq1$, we have the following equivalence results between the iterations of the two algorithms:
\begin{align*}
\vz_\vy^k =  \vt^k,\quad \vz_\vs^k=\vu^k,\quad \vs^k=-\vz_\vu^k,\quad \vy^k=\vz_\vt^k.
\end{align*}
\end{corollary}
The proof is similar to the proof of Theorem~\ref{thm:ADMM-equal} and is omitted here.


\subsection{Example: basis pursuit}\label{sec:examples41}
The basis pursuit problem seeks for the minimal $\ell_1$ solution to a set of linear equations:
\begin{align}\label{for:BP}
\Min_\vu \|\vu\|_1\quad\st~\vA\vu=\vb.
\end{align}
Its Lagrange dual is
\begin{align}\label{for:BPDual}
\Min_\vx -\vb^T\vx\quad\st~\|\vA^*\vx\|_\infty \leq1.
\end{align}
The YALL1 algorithms~\cite{Yall1} implement ADMs on a set of primal and dual formulations for basis pursuit and LASSO, yet ADM for~\eqref{for:BP} is not given (however, a linearized ADM is given for \eqref{for:BP}). Although seemingly awkward, problem ~\eqref{for:BP} can be turned equivalently into the ADM-ready form
\begin{align}\label{for:BPSplt}
\Min_{\vu,\vv} \|\vv\|_1 +\iota_{\{\vu:\vA\vu=\vb\}}(\vu)\quad\st~\vu-\vv=\vzero.
\end{align}
Similarly, problem ~\eqref{for:BPDual} can be turned equivalently into the ADM-ready form
\begin{align}\label{for:BPDualSplt}
\Min_{\vx,\vy} -\vb^T\vx +\iota_{B_1^\infty}(\vy)\quad\st~\vA^*\vx-\vy=\vzero,
\end{align}
where $B_1^\infty=\{\vy:\|\vy\|_\infty\leq1\}$.

For simplicity, let us suppose that $\vA$ has full row rank so the inverse of $\vA\vA^*$ exists. (Otherwise, $\vA\vu=\vb$ are redundant whenever  they are consistent; and $(\vA\vA^*)^{-1}$ shall be replaced by the pseudo-inverse below.) ADM for problem \eqref{for:BPSplt} can be simplified to  the iteration:
\begin{subequations}\label{eqn:BP-P}
\begin{align}
\label{for:BP-PADM-u1}\vv_3^{k+1} = &\argmin_{\vv} \|\vv\|_1+{\lambda\over 2}\|\vu_3^k-\vv+{1\over \lambda}\vz_3^k\|_2^2,\\
\label{for:BP-PADM-v1}\vu_3^{k+1} = &{\vv_3^{k+1}-{1\over \lambda}\vz_3^k}-\vA^*(\vA\vA^*)^{-1}(\vA(\vv_3^{k+1}-{1\over \lambda}\vz_3^k)-\vb),\\
\vz_3^{k+1}=&\vz_3^k+\lambda(\vu_3^{k+1}-\vv_3^{k+1}).
\end{align}
\end{subequations}
And ADM for problem \eqref{for:BPDualSplt} can be simplified to   the iteration:
\begin{subequations}\label{eqn:BP-D}
\begin{align}
\label{for:BP-DADM-y1}\vy_1^{k+1} = &\cP_{B_1^\infty}(\vA^*\vx_1^k+\lambda\vz_1^k),\\
\label{for:BP-DADM-x1}\vx_1^{k+1} = &(\vA\vA^*)^{-1}(\vA\vy_1^{k+1}-\lambda(\vA\vz_1^k-\vb)),\\
\vz_1^{k+1}=&\vz_1^k+\lambda^{-1}(\vA^*\vx_1^{k+1}-\vy_1^{k+1}),
\end{align}
\end{subequations}
where $\cP_{B_1^\infty}$ is the projection onto $B_1^\infty$.
Looking into the iteration in~\eqref{eqn:BP-D}, we can find that $\vA^*\vx_1^k$ is used in both the $k$th and $k+1$st iterations. To save the computation, we can store $\vA^*\vx_1^k$ as $\vs_2^k$. In addition, let $\vt_2^k=\vy_1^k$ and $\vz_2^k=\vz_1^k$, we have 
\begin{subequations}\label{eqn:BP-D2}
\begin{align}
\label{for:BP-DADM-t1}\vt_2^{k+1} = &\cP_{B_1^\infty}(\vs_2^k+\lambda\vz_2^k),\\
\label{for:BP-DADM-s1}\vs_2^{k+1} = &\vA^*(\vA\vA^*)^{-1}(\vA(\vt_2^{k+1}-\lambda\vA\vz_2^k)+\lambda\vb)),\\
\vz_2^{k+1}=&\vz_2^k+\lambda^{-1}(\vs_2^{k+1}-\vt_2^{k+1}),
\end{align}
\end{subequations}
which is exactly Algorithm~\ref{alg:ADM-2} for~\eqref{for:BPDualSplt}. Thus, Algorithm~\ref{alg:ADM-2} has smaller complexity than Algorithm~\ref{alg:ADM-1}, i.e., one matrix vector multiplication $\vA^*\vx_1^k$ is saved from Algorithm~\ref{alg:ADM-2}.

The  corollary below follows directly from Theorem~\ref{thm:ADMM-equal} by associating \eqref{for:BPDualSplt} and  \eqref{for:BPSplt} as \eqref{for:P3-ADMM} and \eqref{for:D4-ADMM}, and \eqref{eqn:BP-D} and \eqref{eqn:BP-P} with the iterations of Algorithms~\ref{alg:ADM-1} and \ref{alg:ADM-3}, respectively.
\begin{corollary}\label{thm:BP-equal}Suppose that $\vA\vu=\vb$ are consistent. Consider ADM iterations \eqref{eqn:BP-P} and  \eqref{eqn:BP-D}. Let $\vu_3^0=\vz_1^0$ and $\vz_3^0=\vA^*\vx_1^0$. Then, for $k\geq1$, iterations \eqref{eqn:BP-P} and \eqref{eqn:BP-D} are equivalent. In particular,
\begin{itemize}
\item From $\vx_1^k$, $\vz_1^k$ in~\eqref{eqn:BP-D}, we obtain $\vu_3^k$, $\vz_3^k$ in~\eqref{eqn:BP-P} through:
\begin{align*}
\vu_3^k=\vz_1^k,\qquad \vz_3^k=\vA^*\vx_1^k.
\end{align*}
\item From $\vu_3^k$, $\vz_3^k$ in~\eqref{eqn:BP-P}, we obtain $\vx_1^k$, $\vz_1^k$ in~\eqref{eqn:BP-D} through:
\begin{align*}
\vx_1^k =(\vA\vA^*)^{-1}\vA\vz_3^k,\qquad
\vz_1^k=\vu_3^k.
\end{align*}
\end{itemize}
\end{corollary}

\subsection{Example: basis pursuit denoising}\label{sec:examples42}
The basis pursuit denoising problem  is
\begin{align}\label{eq:BPD}
\Min_\vu \|\vu\|_1 +{1\over 2\alpha}\|\vA\vu-\vb\|_2^2
\end{align}
and its Lagrange dual, in the ADM-ready form, is
\begin{align}\label{eq:BPDDual}
\Min_{\vx,\vy} -\langle\vb,\vx\rangle  +{\alpha\over2}\|\vx\|_2^2 +\iota_{B_1^\infty}(\vy)\quad\st~ \vA^*\vx-\vy=\vzero.
\end{align}
The iteration of ADM for \eqref{eq:BPDDual} is
\begin{subequations}\label{for:BPDN-DADM}
\begin{align}
\label{for:BPDN-DADM-y1}\vy_1^{k+1} = &\cP_{B_1^\infty}(\vA^*\vx_1^k+\lambda\vz_1^k),\\
\label{for:BPDN-DADM-x1}\vx_1^{k+1} = &(\vA\vA^*+\alpha\lambda\vI)^{-1}(\vA\vy_1^{k+1}-\lambda(\vA\vz_1^k-\vb)),\\
\vz_1^{k+1}=&\vz_1^k+\lambda^{-1}(\vA^*\vx_1^{k+1}-\vy_1^{k+1}).
\end{align}
\end{subequations}
Looking into the iteration in~\eqref{for:BPDN-DADM}, we can find that $\vA^*\vx_1^k$ is used in both the $k$th and $k+1$st iterations. To save the computation, we can store $\vA^*\vx_1^k$ as $\vs_2^k$. In addition, let $\vt_2^k=\vy_1^k$ and $\vz_2^k=\vz_1^k$, we have 
\begin{subequations}\label{for:BPDN-DADM2}
\begin{align}
\label{for:BPDN-DADM-t1}\vt_2^{k+1} = &\cP_{B_1^\infty}(\vs_2^k+\lambda\vz_2^k),\\
\label{for:BPDN-DADM-s1}\vs_2^{k+1} = &\vA^*(\vA\vA^*+\alpha\lambda\vI)^{-1}(\vA(\vt_2^{k+1}-\lambda\vz_2^k)+\lambda\vb)),\\
\vz_2^{k+1}=&\vz_2^k+\lambda^{-1}(\vs_2^{k+1}-\vt_2^{k+1}),
\end{align}
\end{subequations}
which is exactly Algorithm~\ref{alg:ADM-2} for~\eqref{eq:BPDDual}. Thus, Algorithm~\ref{alg:ADM-2} has smaller complexity than Algorithm~\ref{alg:ADM-1}, i.e.,  one matrix vector multiplication $\vA^*\vx_1^k$ is saved from Algorithm~\ref{alg:ADM-2}. In addition, if $\vA^*\vA=\vI$,~\eqref{for:BPDN-DADM-s1} becomes 
\begin{align}
\vs_2^{k+1} = &(\alpha\lambda+1)^{-1}(\vt_2^{k+1}-\lambda\vz_2^k+\lambda\vA^*\vb),
\end{align}
and no matrix vector multiplications is needed during the iteration because $\lambda\vA^*\vb$ can be precalculated.

The ADM-ready form of the original problem \eqref{eq:BPD} is
\begin{align}
\Min_{\vu,\vv} \|\vv\|_1 +{1\over2\alpha}\|\vA\vu-\vb\|_2^2 \quad\st~\vu-\vv=\vzero,
\end{align}
whose ADM iteration is
\begin{subequations}\label{for:BPDN-PADM}
\begin{align}
\label{for:BPDN-PADM-u1}\vv_3^{k+1} = &\argmin_{\vv} \|\vv\|_1+{\lambda\over 2}\|\vu_3^k-\vv+{1\over \lambda}\vz_3^k\|_2^2,\\
\label{for:BPDN-PADM-v1}\vu_3^{k+1} = &(\vA^*\vA+\alpha\lambda\vI)^{-1}(\vA^*\vb+\alpha\lambda\vv_3^{k+1}-\alpha\vz_3^k),\\
\vz_3^{k+1}=&\vz_3^k+\lambda(\vu_3^{k+1}-\vv_3^{k+1}).
\end{align}
\end{subequations}
The  corollary below follows directly from Theorem~\ref{thm:ADMM-equal}.
\begin{corollary}Consider ADM iterations \eqref{for:BPDN-DADM} and \eqref{for:BPDN-PADM}. Let $\vu_3^0=\vz_1^0$ and $\vz_3^0=\vA^*\vx_1^0$. For $k\geq1$, ADM on the dual and primal problems~\eqref{for:BPDN-DADM} and~\eqref{for:BPDN-PADM} are equivalent in the following way:
\begin{itemize}
\item From $\vx_1^k$, $\vz_1^k$ in~\eqref{for:BPDN-DADM}, we recover $\vu_3^k$, $\vz_3^k$ in~\eqref{for:BPDN-PADM} through:
\begin{align*}
\vu_3^k=\vz_1^k, \qquad
\vz_3^k=\vA^*\vx_1^k.
\end{align*}
\item From $\vu_3^k$, $\vz_3^k$ in~\eqref{for:BPDN-PADM}, we recover $\vx_1^k$, $\vz_1^k$ in~\eqref{for:BPDN-DADM} through:
\begin{align*}
\vx_1^k=-(\vA\vu_3^{k}-\vb)/\alpha,\qquad
\vz_1^k=\vu_3^k.\\
\end{align*}
\end{itemize}
\end{corollary}
\begin{remark} Iteration \eqref{for:BPDN-PADM} is different from that of ADM for another ADM-ready form of \eqref{eq:BPD}
\begin{align}
\Min_{\vu,\vv} \|\vu\|_1 +{1\over2\alpha}\|\vv\|_2^2\quad\st~\vA\vu-\vv=\vb,
\end{align}
which is used in~\cite{Yall1}. In general, there are  different ADM-ready forms and their ADM algorithms yield different iterates. ADM on one ADM-ready form is equivalent to it on the corresponding dual ADM-ready form.
\end{remark}

\section{ADM as a primal-dual algorithm on the saddle-point problem}\label{sec:ADM-PD}

As shown in section~\ref{sec:ADM}, ADM on a pair of convex primal and dual problems are equivalent, and there is a connection between $\vz_1^k$ in Algorithm~\ref{alg:ADM-1} and dual variable $\vu_3^k$ in Algorithm~\ref{alg:ADM-3}. This primal-dual equivalence naturally suggests that  ADM is also equivalent to a primal-dual algorithm involving both primal and dual variables.

We derive  problem~\eqref{for:P3-ADMM} into an equivalent primal-dual saddle-point problem \eqref{for:saddle-point} as follows:
\begin{align}
&\min_{\vy,\vx} g(\vy)+f(\vx)+\iota_{\{(\vx,\vy):\vA\vx=\vb-\vB\vy\}}(\vx,\vy)\nonumber\\
=&\min_{\vy} g(\vy)+F(\vb-\vB\vy)\nonumber\\
=&\min_\vy\max_\vu g(\vy)+\langle -\vu, \vb-\vB\vy\rangle -F^*(-\vu)\label{for:saddle-pointb}\\
=&\min_\vy\max_\vu g(\vy)+\langle \vu, \vB\vy-\vb\rangle -f^*(-\vA^*\vu).\label{for:saddle-point}
\end{align}
A primal-dual algorithm for solving  \eqref{for:saddle-point} is described in Algorithm~\ref{alg:primal-dual}. Theorem~\ref{thm:ADM-PD} establishes the equivalence between Algorithms~\ref{alg:ADM-1} and \ref{alg:primal-dual}.

\begin{algorithm}[H]
\caption{Primal-dual formulation of ADM on Problem~\eqref{for:saddle-point}}
\label{alg:primal-dual}
\begin{algorithmic}
\STATE initialize $\vu_4^0$, $\vu_4^{-1}$, $\vy_4^0$, $\lambda>0$
\FOR {$k =0,1, \cdots$}
\STATE $\bar{\vu}_4^k=2\vu_4^{k}-\vu_4^{k-1}$
\STATE $\vy_4^{k+1}=\argmin\limits_\vy g(\vy)+(2\lambda)^{-1}\|\vB\vy-\vB\vy_4^k+\lambda\bar\vu_4^k\|_2^2$
\STATE $\vu_4^{k+1} =\argmin\limits_\vu f^*(-\vA^*\vu)-\langle\vu,\vB\vy_4^{k+1}-\vb\rangle+\lambda/2\|\vu-\vu_4^k\|_2^2$
\ENDFOR
\end{algorithmic}
\end{algorithm}

\begin{remark}
Paper~\cite{chambolle2011first} proposed a primal-dual algorithm for~\eqref{for:saddle-pointb} and obtained the connection between ADM and that primal-dual algorithm~\cite{esser2010a}: When $\vB=\vI$, ADM is equivalent to the primal-dual algorithm in~\cite{chambolle2011first}; When $\vB\neq\vI$, the primal-dual algorithm is a preconditioned ADM as an additional proximal term $\delta/2\|\vy-\vy_4^k\|_2^2-(2\lambda)^{-1}\|\vB\vy-\vB\vy_4^k\|_2^2$ is added to the subproblem for $\vy_4^{k+1}$. This is also a special case of inexact ADM in~\cite{DengYin2012a}. Our Algorithm \ref{alg:primal-dual} is a primal-dual algorithm that is equivalent to ADM in the general case.
\end{remark}

\begin{theorem}[Equivalence between Algorithms~\ref{alg:ADM-1} and \ref{alg:primal-dual}] \label{thm:ADM-PD}
Suppose that $\vA\vx_1^0=\lambda(\vu_4^0-\vu_4^{-1})+\vb-\vB\vy_4^0$ and $\vz_1^0=\vu_4^0$. Then, Algorithms~\ref{alg:ADM-1} and~\ref{alg:primal-dual} are equivalent with the identities:
\begin{align}\label{for:Axuub}
\vA\vx_1^k=\lambda(\vu_4^k-\vu_4^{k-1})+\vb-\vB\vy_4^k,\qquad \vz_1^k=\vu_4^k,
\end{align} for all $k>0$.
\end{theorem}

\begin{proof}By assumption, \eqref{for:Axuub} holds at iteration $k=0$. \\
Proof by induction. Suppose that \eqref{for:Axuub} holds at iteration $k\ge 0$. We shall establish \eqref{for:Axuub} at iteration $k+1$. From the first step of Algorithm~\ref{alg:ADM-1}, we have
\begin{align*}
\vy_1^{k+1}=&\argmin_\vy g(\vy)+(2\lambda)^{-1}\|\vA\vx_1^k+\vB\vy-\vb+\lambda\vz_1^k\|_2^2\\
=&\argmin_\vy g(\vy)+(2\lambda)^{-1}\|\lambda(\vu_4^k-\vu_4^{k-1})+\vB\vy-\vB\vy_4^{k}+\lambda\vu_4^k\|_2^2,
\end{align*}
which is the same as the first step in Algorithm~\ref{alg:primal-dual}. Thus we have $\vy_1^{k+1}=\vy_4^{k+1}$.

Combing the second and third steps of Algorithm~\ref{alg:ADM-1}, we have
\begin{align*}
\vzero \in \partial f(\vx_1^{k+1})+\lambda^{-1}\vA^*(\vA\vx_1^{k+1}+\vB\vy_1^{k+1}-\vb+\lambda\vz_1^k)=\partial f(\vx_1^{k+1}) + \vA^*\vz_1^{k+1}.
\end{align*}
Therefore,
\begin{align*}
&\vx_1^{k+1}\in\partial f^*(-\vA^*\vz_1^{k+1}) \\
\Longrightarrow~& \vA\vx_1^{k+1}\in\partial F^*(-\vz_1^{k+1})\\
\Longleftrightarrow~& \lambda (\vz_1^{k+1}-\vz_1^{k})+\vb-\vB\vy_1^{k+1}\in\partial F^*(-\vz_1^{k+1})\\
\Longleftrightarrow~&\vz_1^{k+1} =\argmin_{\vz} F^*(-\vz)-\langle\vz,\vB\vy_1^{k+1}-\vb\rangle+\lambda/2\|\vz-\vz_1^k\|_2^2\\
\Longleftrightarrow~&\vz_1^{k+1} =\argmin_{\vz} f^*(-\vA^*\vz)-\langle\vz,\vB\vy_4^{k+1}-\vb\rangle+\lambda/2\|\vz-\vu_4^k\|_2^2,
\end{align*}
where the last line is the second step of Algorithm~\ref{alg:primal-dual}. Therefore, we have $\vz_1^{k+1}=\vu_4^{k+1}$ and $\vA\vx_1^{k+1}=\lambda (\vz_1^{k+1}-\vz_1^{k})+\vb-\vB\vy_1^{k+1}=\lambda (\vu_4^{k+1}-\vu_4^{k})+\vb-\vB\vy_4^{k+1}$.
\end{proof}

\section{Equivalence of ADM for different orders}\label{sec:ADM-equal2}

In both problem \eqref{for:P3-ADMM} and Algorithm ~\ref{alg:ADM-1}, we can swap  $\vx$ and $\vy$ and obtain Algorithm \ref{alg:ADM-4}, which is still an algorithm of ADM. In general, the two algorithms are different. In this section, we show that for a certain type of functions $f$ (or $g$),  Algorithms~\ref{alg:ADM-1} and \ref{alg:ADM-4} become equivalent.

\begin{algorithm}[H]
\caption{ADM2 on~\eqref{for:P3-ADMM}}
\label{alg:ADM-4}
\begin{algorithmic}
\STATE initialize $\vy_5^0$, $\vz_5^0$, $\lambda>0$
\FOR {$k =0,1, \cdots$}
\STATE $\vx_5^{k+1} = \argmin\limits_\vx f(\vx)+(2\lambda)^{-1}\|\vA\vx+\vB\vy_5^{k}-\vb+\lambda\vz_5^k\|_2^2$
\STATE $\vy_5^{k+1} = \argmin\limits_\vy g(\vy)+(2\lambda)^{-1}\|\vA\vx_5^{k+1}+\vB\vy-\vb+\lambda\vz_5^k\|_2^2$
\STATE $\vz_5^{k+1} = \vz_5^k+\lambda^{-1}(\vA\vx_5^{k+1}+\vB\vy_5^{k+1}-\vb)$
\ENDFOR
\end{algorithmic}
\end{algorithm}

The assumption that we need is that either $\prox_{ F(\cdot)}$ or $\prox_{ G(\cdot)}$ is affine (cf. \eqref{for:subprobs} for the definitions of $F$ and $G$).
\begin{definition} A mapping $T$ is affine if, for any $\vr_1$ and $\vr_2$,
\begin{align*}
T\left(\frac{1}{2}\vr_1+\frac{1}{2}\vr_2\right)=\frac{1}{2}T\vr_1+\frac{1}{2}T\vr_2.
\end{align*}
\end{definition}
\begin{proposition} Let $\lambda>0$. The following statements are equivalent:
\begin{enumerate}
\item $\prox_{G(\cdot)}$ is affine;
\item $\prox_{\lambda G(\cdot)}$ is affine;
\item $a\prox_{G(\cdot)}\circ b\vI+c\vI$ is affine for any scalars $a$, $b$ and $c$;
\item $\prox_{G^*(\cdot)}$ is affine;
\item {$G$ is convex quadratic (or, affine or constant) and its domain $\dom(G)$  is either $\cG$ or the intersection of hyperplanes in $\cG$.}
\end{enumerate}
{In addition, if function $g$ is convex quadratic and its domain is the intersection of hyperplanes, then function $G$ defined in \eqref{for:subprobG} satisfies Part 5 above.}
\end{proposition}

\begin{proposition}If $\prox_{G(\cdot)}$ is affine, then the following holds for any $\vr_1$ and $\vr_2$:
\begin{align}\label{for:prxg}
\prox_{G(\cdot)}(2\vr_1-\vr_2)=2\prox_{G(\cdot)}\vr_1-\prox_{G(\cdot)}\vr_2.
\end{align}
\end{proposition}
\begin{proof}
Equation \eqref{for:prxg} is obtained by defining $\bar{\vr}_1=2\vr_1-\vr_2$ and $\bar{\vr}_2:=\vr_2$ and rearranging
\begin{align*}\prox_{G(\cdot)}\left(\frac{1}{2}\bar{\vr}_1 +\frac{1}{2}\bar{\vr}_2\right) =\frac{1}{2}\prox_{G(\cdot)}\bar{\vr}_1+\frac{1}{2}\prox_{G(\cdot)}\bar{\vr}_2. \end{align*}
\end{proof}

\begin{theorem}[Equivalence of Algorithms~\ref{alg:ADM-1} and~\ref{alg:ADM-4}] \label{thm:ADMM-equal2}
\ \\
\begin{enumerate}
\item \label{swaporder_equiv_part1}Assume that $\prox_{\lambda G(\cdot)}$ is affine. Given the sequences $\vy_5^{k}$, $\vz_5^{k}$, and $\vx_5^{k}$ of Algorithm~\ref{alg:ADM-4}, if $\vy_5^0$ and $\vz_5^0$ satisfy $-\vz_5^0\in \partial G(\vB\vy_5^0-\vb)$, then
we can initialize Algorithm~\ref{alg:ADM-1} with $\vx_1^0=\vx_5^1$ and $\vz_1^0=\vz_5^0+\lambda^{-1}(\vA\vx_5^{1}+\vB\vy_5^{0}-\vb)$, and recover the sequences  $\vx_1^k$ and $\vz_1^k$ of Algorithm~\ref{alg:ADM-1} through
\begin{subequations}\label{for:ADM-order}
\begin{align}
\vx_1^k&=\vx_5^{k+1},\\
\vz_1^k&=\vz_5^{k}+\lambda^{-1}(\vA\vx_5^{k+1}+\vB\vy_5^{k}-\vb).
\end{align}
\end{subequations}
\item Assume that $\prox_{\lambda F(\cdot)}$ is affine. Given the sequences  $\vx_1^{k}$, $\vz_1^{k}$, and $\vy_1^{k}$ of Algorithm~\ref{alg:ADM-1}, if $\vx_1^0$ and $\vz_1^0$ satisfy $-\vz_1^0\in \partial F(\vA\vx_1^0)$, then we can initialize Algorithm~\ref{alg:ADM-4} with $\vy_5^0=\vy_1^1$ and $\vz_5^0=\vz_1^0+\lambda^{-1}(\vA\vx_1^{0}+\vB\vy_1^{1}-\vb)$, and recover the sequences $\vy_5^k$ and $\vz_5^k$ of Algorithm~\ref{alg:ADM-4} through
\begin{subequations}\label{for:ADM-order2}
\begin{align}
\vy_5^k&=\vy_1^{k+1},\\
\vz_5^k&=\vz_1^{k}+\lambda^{-1}(\vA\vx_1^{k}+\vB\vy_1^{k+1}-\vb).
\end{align}
\end{subequations}
\end{enumerate}
\end{theorem}
\begin{proof} We prove Part \ref{swaporder_equiv_part1} only by induction. (The proof for the other part is similar.)
The initialization of Algorithm~\ref{alg:ADM-1} clearly follows~\eqref{for:ADM-order} at $k=0$. Suppose that~\eqref{for:ADM-order} holds at $k\ge 0$. We shall show that~\eqref{for:ADM-order} holds at $k+1$. We first show from the affine property of $\prox_{\lambda G(\cdot)}$ that \begin{align}\label{for:thm2-y}
\vB\vy_1^{k+1}=2\vB\vy_5^{k+1}-\vB\vy_5^{k}.
\end{align}
The optimization subproblems for $\vy_1$ and $\vy_5$ in Algorithms~\ref{alg:ADM-1} and~\ref{alg:ADM-4}, respectively, are as follows:
\begin{align*}
\vy_1^{k+1} &= \argmin\limits_\vy g(\vy)+(2\lambda)^{-1}\|\vA\vx_1^{k}+\vB\vy-\vb+\lambda \vz_1^k\|_2^2,\\
\vy_5^{k+1} &= \argmin\limits_\vy g(\vy) +(2\lambda)^{-1}\|\vA\vx_5^{k+1}+\vB\vy-\vb+\lambda \vz_5^k\|_2^2.
\end{align*}
Following the definition of $G$ in \eqref{for:subprobs}, we have
\begin{subequations}\label{for:bBy}
\begin{align}
\vB\vy_1^{k+1}&-\vb = \prox_{\lambda G(\cdot)}(-\vA\vx_1^{k}-\lambda \vz_1^k),\\
\vB\vy_5^{k+1}&-\vb = \prox_{\lambda G(\cdot)}(-\vA\vx_5^{k+1}-\lambda \vz_5^k),\\
\vB\vy_5^{k}\quad&-\vb = \prox_{\lambda G(\cdot)}(-\vA\vx_5^{k}-\lambda \vz_5^{k-1}).\label{for:bByc}
\end{align}
\end{subequations}
The third step of Algorithm~\ref{alg:ADM-4} is
\begin{align}\label{for:alg4step3}\vz_5^{k}= \vz_5^{k-1}+\lambda^{-1}(\vA\vx^k_5+\vB\vy^k_5-\vb).\end{align}
(Note that for $k=0$, the assumption $-\vz_5^0\in \partial G(\vB\vy_5^0-\vb)$ ensures the existence of $\vz_5^{-1}$ in \eqref{for:bByc} and \eqref{for:alg4step3}.)
Then, \eqref{for:ADM-order} and \eqref{for:alg4step3}  give us
\begin{align*}
\vA\vx_1^{k}+\lambda\vz_1^{k}&\overset{\eqref{for:ADM-order}}{=}\vA\vx_5^{k+1}+\lambda\vz_5^{k}+\vA\vx_5^{k+1}+\vB\vy_5^{k}-\vb\\
&\ =\ 2(\vA\vx_5^{k+1}+\lambda\vz_5^{k})-(\lambda\vz_5^{k}-\vB\vy_5^{k}+\vb)\\
&\overset{\eqref{for:alg4step3}}{=}\ 2(\vA\vx_5^{k+1}+\lambda\vz_5^{k})-(\vA\vx_5^{k}+\lambda\vz_5^{k-1}).
\end{align*}
Since  $\prox_{\lambda G(\cdot)}$ is affine, we have \eqref{for:prxg}. Once we plug in \eqref{for:prxg}: $\vr_1=-\vA\vx_5^{k+1}-\lambda\vz_5^{k}$, $\vr_2=-\vA\vx_5^{k}-\lambda\vz_5^{k-1}$, and $2\vr_1-\vr_2 = -\vA\vx_1^{k}-\lambda\vz_1^k$ and then apply  \eqref{for:bBy}, we obtain~\eqref{for:thm2-y}.

Next, the third step of Algorithm~\ref{alg:ADM-4} and \eqref{for:thm2-y} give us
\begin{align*}
\vB\vy_1^{k+1}-\vb+\lambda \vz_1^k &\overset{\eqref{for:thm2-y}}{=} 2(\vB\vy_5^{k+1}-\vb)-(\vB\vy_5^{k}-\vb)+\lambda\vz_5^{k}+(\vA\vx_5^{k+1}+\vB\vy_5^{k}-\vb)\\
&\ =\ (\vB\vy_5^{k+1}-\vb)+\lambda\vz_5^{k}+(\vA\vx_5^{k+1}+\vB\vy_5^{k+1}-\vb)\\
&\ =\ (\vB\vy_5^{k+1}-\vb)+\lambda\vz_5^{k+1}.
\end{align*}
This identity shows that the updates of  $\vx_1^{k+1}$ and $\vx_5^{k+2}$ in Algorithms~\ref{alg:ADM-1} and~\ref{alg:ADM-4}, respectively, have identical data, and therefore, we recover  $\vx_1^{k+1}=\vx_5^{k+2}$.

Lastly, from the third step of Algorithm~\ref{alg:ADM-1} and the identities above, it follows that
\begin{align*}
\vz_1^{k+1}&=\vz_1^k+\lambda^{-1}(\vA\vx_1^{k+1}+\vB\vy_1^{k+1}-\vb)\\
&=\vz_1^k+\lambda^{-1}\left(\vA\vx_5^{k+2}+(\vB\vy_5^{k+1}-\vb+\lambda\vz_5^{k+1}-\lambda\vz_1^k)\right)\\
&=\vz_5^{k+1}+\lambda^{-1}(\vA\vx_5^{k+2}+\vB\vy_5^{k+1}-\vb).
\end{align*}
Therefore, we obtain \eqref{for:ADM-order} at $k+1$.
\end{proof}
\begin{remark} We can avoid the technical condition $-\vz_5^0\in \partial G(\vB\vy_5^0-\vb)$ on Algorithm~\ref{alg:ADM-4} in Part~\ref{swaporder_equiv_part1} of Theorem \ref{thm:ADMM-equal2}. When it does not hold, we can use the always-true relation $-\vz_5^1\in \partial G(\vB\vy_5^1-\vb)$ instead; correspondingly, we shall add 1 iteration to the iterates of Algorithm~\ref{alg:ADM-4}, namely,  initialize Algorithm~\ref{alg:ADM-1} with $\vx_1^0=\vx_5^2$ and $\vz_1^0=\vz_5^1+\lambda^{-1}(\vA\vx_5^{2}+\vB\vy_5^{1}-\vb)$ and recover the sequences  $\vx_1^k$ and $\vz_1^k$ of Algorithm~\ref{alg:ADM-1} through
\begin{subequations}\label{for:ADM-order_offset}
\begin{align}
\vx_1^k&=\vx_5^{k+2},\\
\vz_1^k&=\vz_5^{k+1}+\lambda^{-1}(\vA\vx_5^{k+2}+\vB\vy_5^{k+1}-\vb).
\end{align}
\end{subequations}
Similar arguments apply to the other part of Theorem \ref{thm:ADMM-equal2}.
\end{remark}

\section{Primal-dual equivalence of RPRS}
\label{sec:DRS}


In this section, we consider the following convex problem:
\begin{align}\tag{P3}\label{for:P1-DRS}
\Min_\vx \quad f(\vx)+g(\vA\vx),
\end{align}
and its corresponding Lagrangian dual
\begin{align}\tag{D3}\label{for:D1-DRS}
\Min_\vv  \quad f^*(\vA^*\vv) + g^*(-\vv).
\end{align}
In addition, we introduce another primal-dual pair equivalent to~\eqref{for:P1-DRS}-\eqref{for:D1-DRS}:
\begin{align}\tag{P4}\label{for:P2-DRS}
\Min_\vy &\quad (f^*\circ \vA^*)^*(\vy)+g(\vy),\\
 \tag{D4}\label{for:D2-DRS}
\Min_\vu  &\quad f^*(\vu) + (g\circ \vA)^*(-\vu).
\end{align}
Lemma~\ref{lemma:Equivalency} below will establish the equivalence between the two primal-dual pairs.

\begin{remark}When $\vA=\vI$, we have $(f^*\circ \vA^*)^*=f$, and problem~\eqref{for:P1-DRS} is exactly the same as problem~\eqref{for:P2-DRS}. Similarly, problem~\eqref{for:D1-DRS} is exactly the same as problem~\eqref{for:D2-DRS}.
\end{remark}

\begin{lemma}\label{lemma:Equivalency}
Problems~\eqref{for:P1-DRS} and~\eqref{for:P2-DRS} are equivalent in the following sense:
\begin{itemize}
\item Given any solution $\vx^*$ to~\eqref{for:P1-DRS}, $\vy^*=\vA\vx^*$ is a solution to~\eqref{for:P2-DRS},
\item Given any solution $\vy^*$ to~\eqref{for:P2-DRS}, $\vx^*\in\argmin\limits_{\vx:\vA\vx=\vy^*} f(\vx)$ is a solution to~\eqref{for:P1-DRS}.
\end{itemize}
The equivalence between problems~\eqref{for:D1-DRS} and~\eqref{for:D2-DRS} is similar:
\begin{itemize}
\item Given any solution $\vv^*$ to~\eqref{for:D1-DRS}, $\vA^*\vv^*$ is a solution to~\eqref{for:D2-DRS},
\item Given any solution $\vu^*$ to~\eqref{for:D2-DRS}, $\vv^*\in\argmin\limits_{\vv:\vA^*\vv=\vu^*} g^*(-\vv)$ is a solution to~\eqref{for:D1-DRS}.
\end{itemize}
\end{lemma}

\begin{proof}
We prove only the equivalence of~\eqref{for:P1-DRS} and~\eqref{for:P2-DRS}, the proof for the equivalence of~\eqref{for:D1-DRS} and~\eqref{for:D2-DRS} is similar.

Part 1: If $\vx^*$ is a solution to~\eqref{for:P1-DRS}, we have $\vzero\in\partial f(\vx^*)+\vA^*\partial g(\vA\vx^*)$. Assume that there exists $\vq$ such that $-\vq\in\partial g(\vA\vx^*)$ and $\vA^*\vq\in\partial f(\vx^*)$. Then we have
\begin{align*}
\vA^*\vq\in\partial f(\vx^*) \Longleftrightarrow & \vx^*\in\partial f^*(\vA^*\vq)\\
\Longrightarrow & \vA\vx^* \in \vA\partial f^*(\vA^*\vq)= \partial (f^*\circ \vA^*)(\vq)\\
\Longleftrightarrow & \vq \in \partial(f^*\circ \vA^*)^*(\vA\vx^*).
\end{align*}
Therefore,
\begin{align*}
\vzero \in  \partial(f^*\circ \vA^*)^*(\vA\vx^*) +\partial g(\vA\vx^*)
\end{align*}
and $\vA\vx^*$ is a solution to~\eqref{for:P2-DRS}.

Part 2: If $\vy^*$ is a solution to~\eqref{for:P2-DRS}, the optimality condition gives us
\begin{align*}
\vzero \in  \partial(f^*\circ \vA^*)^*(\vy^*) +\partial g(\vy^*).
\end{align*}
Assume that there exists $\vq$ such that $-\vq\in\partial g(\vy^*)$ and $\vq\in \partial(f^*\circ \vA^*)^*(\vy^*)$. Then we have
\begin{align}\label{for:PD-equal-1}
\vq \in \partial(f^*\circ \vA^*)^*(\vy^*)\Longleftrightarrow &\vy^* \in  \partial (f^*\circ \vA^*)(\vq).
\end{align}
Consider the following optimization problem for finding $\vx^*$ from $\vy^*$
\begin{align*}
\Min_\vx f(\vx)\qquad\mbox{subject to }\vA\vx=\vy^*,
\end{align*}
and the corresponding dual problem
\begin{align*}
\Max_\vv -f^*(\vA^*\vv)+\langle\vv,\vy^*\rangle.
\end{align*}
It is easy to obtain from~\eqref{for:PD-equal-1} that $\vq$ is a solution of the dual problem. The optimal duality gap is zero and the strong duality gives us
\begin{align}
f(\vx^*)=f(\vx^*)-\langle \vq,\vA\vx^*-\vy^*\rangle =  -f^*(\vA^*\vq)+\langle \vq,\vy^*\rangle.
\end{align}
Thus $\vx^*$ is a solution of $\Min\limits_\vx f(\vx)-\langle \vA^*\vq, \vx\rangle$ and
\begin{align}
\vA^*\vq \in \partial f(\vx^*)\Longleftrightarrow \vzero\in \partial f(\vx^*)-\vA^*\vq.
\end{align}
Because $-\vq\in\partial g(\vy^*)=\partial g(\vA\vx^*)$,
\begin{align}
\vzero\in \partial f(\vx^*)+\vA^*\partial\vg(\vA\vx^*)= \partial f(\vx^*)+\partial(\vg\circ \vA)(\vx^*)
\end{align}
Therefore $\vx^*$ is a solution of~\eqref{for:P1-DRS}.
\end{proof}

Next we will show the equivalence between the RPRS to the primal and dual problems:
\begin{align*}
\boxed{\mbox{RPRS on~\eqref{for:P1-DRS}}}\Longleftrightarrow \boxed{\mbox{RPRS on~\eqref{for:D2-DRS}}}\\
\boxed{\mbox{RPRS on~\eqref{for:P2-DRS}}}\Longleftrightarrow \boxed{\mbox{RPRS on~\eqref{for:D1-DRS}}}
\end{align*}
We describe the RPRS on~\eqref{for:P1-DRS} in Algorithm~\ref{alg:DRS-P4}, and the RPRS on other problems can be obtained in the same way.
\begin{algorithm}[H]
\caption{RPRS on~\eqref{for:P1-DRS}}
\label{alg:DRS-P4}
\begin{algorithmic}
\STATE initialize $\vw^0$, $\lambda>0$, $0<\alpha\leq1$.
\FOR {$k =0,1, \cdots$}
\STATE $\vx^{k+1}  = \prox_{\lambda f(\cdot)} \vw^k$
\STATE $\vw^{k+1}=(1-\alpha) \vw^k+\alpha(2\prox_{\lambda g\circ \vA(\cdot)}-\vI)(2\vx^{k+1}-\vw^k)$
\ENDFOR
\end{algorithmic}
\end{algorithm}

\begin{theorem}\label{thm:DRS-equal}[Primal-dual equivalence of RPRS] RPRS on~\eqref{for:P1-DRS} is equivalent to RPRS on~\eqref{for:D2-DRS}. RPRS on~\eqref{for:P2-DRS} is equivalent to RPRS on~\eqref{for:D1-DRS}.
\end{theorem}

Before proving this theorem, we introduce a lemma, which was also given in~\cite[Proposition~3.34]{eckstein1989splitting}. Here, we prove it in a different way using the generalized Moreau decomposition.
\begin{lemma}\label{lemma:2}
For $\lambda>0$, we have
\begin{align}\label{for:PD-lemma}
\lambda^{-1}(2\prox_{\lambda F(\cdot)}-\vI)\vw = (\vI-2\prox_{\lambda^{-1}F^*(\cdot)})(\vw/\lambda)= (2\prox_{\lambda^{-1}F^*(-\cdot)}-\vI)(-\vw/\lambda).
\end{align}
\end{lemma}

\begin{proof}
We prove it using the generalized Moreau decomposition~\cite[Theorem 2.3.1]{esser_primal_2010}
\begin{align}\label{for:Moreau-decomp}
\vw = \prox_{\lambda F(\cdot)}(\vw) + \lambda \prox_{\lambda^{-1}F^*(\cdot)}(\vw/\lambda).
\end{align}
Using the generalized Moreau decomposition, we have
\begin{align*}
\lambda^{-1}(2\prox_{\lambda F(\cdot)}-\vI)\vw =&2\lambda^{-1}\prox_{\lambda F(\cdot)}(\vw)-\vw/\lambda\overset{\eqref{for:Moreau-decomp}}{=}2\lambda^{-1}(\vw-\lambda\prox_{\lambda^{-1}F^*(\cdot)}(\vw/\lambda))-\vw/\lambda\\
=&\vw/\lambda-2\prox_{\lambda^{-1}F^*(\cdot)}(\vw/\lambda) = (\vI-2\prox_{\lambda^{-1}F^*(\cdot)})(\vw/\lambda).
\end{align*}
The last equality of~\eqref{for:PD-lemma} comes from
\begin{align*}
\prox_{\lambda^{-1}F^*(-\cdot)}(-\vw/\lambda) = -\prox_{\lambda^{-1}F^*(\cdot)}(\vw/\lambda).
\end{align*}
\end{proof}

\begin{proof}[Proof of Theorem~\ref{thm:DRS-equal}]
We will prove only the equivalence of RPRS on~\eqref{for:P1-DRS} and~\eqref{for:D2-DRS}. The proof for the other equivalence is the same. The RPRS on~\eqref{for:P1-DRS} and~\eqref{for:D2-DRS} can be formulated as
\begin{align}\label{for:P-DRS}
\vw_1^{k+1}&=(1-\alpha) \vw_1^k+\alpha(2\prox_{\lambda g\circ \vA(\cdot)}-\vI)(2\prox_{\lambda f(\cdot)}-\vI)\vw_1^k,
\end{align}
and
\begin{align}\label{for:D-DRS}
\vw_2^{k+1}&=(1-\alpha) \vw_2^k+\alpha(2\prox_{\lambda^{-1}(g\circ \vA)^*(-\cdot)}-\vI)(2\prox_{\lambda^{-1}f^*(\cdot)}-\vI)\vw_2^k,
\end{align}
respectively. In addition, we can recover the variables $\vx^k$ (or $\vv^k$) from $\vw_1^k$ (or $\vw_2^k$) using the following forms:
\begin{align}
\vx^{k+1} & = \prox_{\lambda f(\cdot)} \vw_1^{k},\\
\vv^{k+1} & = \prox_{\lambda^{-1}f^*(\cdot)} \vw_2^{k}.
\end{align}
Proof by induction. Suppose $\vw_2^k = \vw_1^k/\lambda$ holds. We next show that $\vw_2^{k+1} = \vw_1^{k+1}/\lambda$.
\begin{align*}
\vw_2^{k+1}=&(1-\alpha) \vw_1^k/\lambda+\alpha(2\prox_{\lambda^{-1}(g\circ \vA)^*(-\cdot)}-\vI)(2\prox_{\lambda^{-1}f^*(\cdot)}-\vI)(\vw_1^k/\lambda)\\
\overset{\eqref{for:PD-lemma}}{=}&(1-\alpha) \vw_1^k/\lambda+\alpha(2\prox_{\lambda^{-1}(g\circ \vA)^*(-\cdot)}-\vI)(-\lambda^{-1}(2\prox_{\lambda f(\cdot)}-\vI)\vw_1^k)\\
\overset{\eqref{for:PD-lemma}}{=}&(1-\alpha) \vw_1^k/\lambda+\alpha\lambda^{-1}(2\prox_{\lambda(g\circ \vA)(-\cdot)}-\vI)(2\prox_{\lambda f(\cdot)}-\vI)\vw_1^k\\
=&\lambda^{-1}[(1-\alpha) \vw_1^k+\alpha(2\prox_{\lambda(g\circ \vA)(\cdot)}-\vI)(2\prox_{\lambda f(\cdot)}-\vI)\vw_1^k]\\
=&\vw_1^{k+1}/\lambda.
\end{align*}
In addition we have
\begin{align*}
\vx^{k+1} +\lambda\vv^{k+1}=&\prox_{\lambda f(\cdot)} \vw_1^{k}+\lambda\prox_{\lambda^{-1}f^*(\cdot)} \vw_2^{k}\\
=&\prox_{\lambda f(\cdot)} \vw_1^{k}+\lambda\prox_{\lambda^{-1}f^*(\cdot)} (\vw_1^{k}/\lambda)=\vw_1^{k}.
\end{align*}
\end{proof}

\begin{remark}Eckstein showed in~\cite[Chapter~3.5]{eckstein1989splitting} that DRS/PRS on~\eqref{for:P1-DRS} is equivalent to DRS/PRS on~\eqref{for:D1-DRS} when $\vA=\vI$. This special case can be obtained from this theorem immediately because when $\vA=\vI$, \eqref{for:D1-DRS} is exactly the same as~\eqref{for:D2-DRS} and we have
\begin{align*}
\boxed{\textnormal{DRS/PRS on~\eqref{for:P1-DRS}}}\Longleftrightarrow \boxed{\textnormal{DRS/PRS on~\eqref{for:D2-DRS}}}\Longleftrightarrow \boxed{\textnormal{DRS/PRS on~\eqref{for:D1-DRS}}}\Longleftrightarrow\boxed{\textnormal{DRS/PRS on~\eqref{for:P2-DRS}}}.
\end{align*}
\end{remark}

\begin{remark}In order to make sure that RPRS on the primal and dual problems are equivalent, the initial conditions and parameters have to satisfy conditions described in the proof of Theorem~\ref{thm:DRS-equal}. We need the initial condition to satisfy $\vw_2^0=\vw_1^0/\lambda$ and the parameter for RPRS on the dual problem has to be chosen as $\lambda^{-1}$, see the differences in~\eqref{for:P-DRS} and~\eqref{for:D-DRS}.
\end{remark}

\section{Application: total variation image denoising}
\label{sec:app1}

ADM (or split Bregman~\cite{SplitBregman}) has been applied on many image processing applications, and we apply the previous equivalence results of ADM to derive several equivalent algorithms for total variation denoising.

The total variation (ROF model~\cite{ROF}) applied on image denoising is
\begin{align*}
\Min_{x\in BV(\Omega)} \int_\Omega |D x|+ {\alpha\over2}\| x-b\|_2^2
\end{align*}
where $x$ stands for an image, and $BV(\Omega)$ is the set of all bounded variation functions on $\Omega$. The first term is known as the total variation of $x$, minimizing which tends to yield a piece-wise constant solution. The discrete version is as follows:
\begin{align*}
\Min_{\vx}  \|\nabla \vx\|_{2,1}+ {\alpha\over2}\| \vx-\vb\|_2^2.
\end{align*}
Without loss of generality, we consider the two-dimensional image $\vx$, and the discrete total variation $\|\nabla\vx\|_{2,1}$ of image $\vx$ is defined as
\begin{align*}
\|\nabla\vx\|_{2,1}=\sum_{ij}|(\nabla \vx)_{ij}|,
\end{align*}
where $|\cdot|$ is the 2-norm of a vector. The equivalent ADM-ready form~\cite[Equation (3.1)]{SplitBregman} is
\begin{align}\label{TVD-P}
\Min_{\vx,\vy}  \ \|\vy\|_{2,1}+ {\alpha\over2}\| \vx-\vb\|_2^2\qquad\st \ \vy-\nabla \vx=\vzero,
\end{align}
and its dual problem in ADM-ready form~\cite[Equation (8)]{Chambolle} is
\begin{align}\label{TVD-D}
\Min_{\vv,\vu}  \ {1\over2\alpha}\| \textnormal{div }\vu+\alpha\vb\|_2^2+\iota_{\{\vv:{\|\vv\|_{2,\infty}\leq1}\}}(\vv)\qquad \st \ \vu-\vv=\vzero,
\end{align}
where $\|\vv\|_{2,\infty}=\max\limits_{ij}|(\vv)_{ij}|$. In addition, the equivalent saddle-point problem is
\begin{align}\label{TVD-PD}
\Min_{\vx}\Max_{\vv}\quad  {1\over2\alpha}\|\vx-\vb\|_2^2+\langle \vv,\nabla \vx\rangle-\iota_{\{\vv:{\|\vv\|_{2,\infty}\leq1}\}}(\vv).
\end{align}

We list the following equivalent algorithms for solving the total variation image denoising problem. The equivalence result stated in Corollary~\ref{cor:TV} can be obtained from theorems~\ref{thm:ADMM-equal}-\ref{thm:ADMM-equal2}.
\begin{enumerate}
\item Algorithm~\ref{alg:ADM-1} (primal ADM) on~\eqref{TVD-P} is
\begin{subequations}\label{eqn:TV-P}
\begin{align}
\vx_1^{k+1} = &\argmin_{\vx}{\alpha\over2}\|\vx-\vb\|_2^2+(2\lambda)^{-1}\|\nabla\vx-\vy_1^k+\lambda\vz_1^k\|_2^2,\\
\vy_1^{k+1} = &\argmin_{\vy}\|\vy\|_{2,1}+(2\lambda)^{-1}\|\nabla\vx_1^{k+1}-\vy+\lambda\vz_1^k\|_2^2,\\
\vz_1^{k+1}=&\vz_1^k+\lambda^{-1}(\nabla\vx_1^{k+1}-\vy_1^{k+1}).
\end{align}
\end{subequations}

\item Algorithm~\ref{alg:ADM-3} (dual ADM) on~\eqref{TVD-D} is
\begin{subequations}\label{eqn:TV-D}
\begin{align}
\vu_3^{k+1} = &\argmin_{\vu}{{1\over2\alpha}\|\textnormal{div }\vu+\alpha\vb\|_2^2+{\lambda\over2}\|\vv_3^k-\vu+\lambda^{-1}\vz_3^k\|_2^2},\\
\vv_3^{k+1} = &\argmin_{\vv}\iota_{\{\vv:\|\vv\|_{2,\infty}\leq1\}}(\vv)+{\lambda\over 2}\|\vv-\vu_3^{k+1}+\lambda^{-1}\vz_3^k\|_2^2,\\
\vz_3^{k+1} = & \vz_3^k+\lambda(\vv_3^{k+1}-\vu_3^{k+1}).
\end{align}
\end{subequations}

\item Algorithm~\ref{alg:primal-dual} (primal-dual) on~\eqref{TVD-PD} is
\begin{subequations}\label{eqn:TV-PD}
\begin{align}
\bar\vv_4^{k} = & 2\vv_4^{k}-\vv_4^{k-1},\\
\vx_4^{k+1} = &\argmin_{\vx}{\alpha\over2}\|\vx-\vb\|_2^2+(2\lambda)^{-1}\|\nabla\vx-\nabla \vx_4^k+\lambda\bar\vv_4^k\|_2^2,\\
\vv_4^{k+1} = &\argmin_{\vv}\iota_{\{\vv:\|\vv\|_{2,\infty}\leq1\}}(\vv)-\langle\vv,\nabla\vx_4^{k+1}\rangle+{\lambda\over 2}\|\vv-\vv^k\|_2^2.
\end{align}
\end{subequations}

\item Algorithm~\ref{alg:ADM-4} (primal ADM with order swapped) on~\eqref{TVD-P} is
\begin{subequations}\label{eqn:TV-P2}
\begin{align}
\vy_5^{k+1} = &\argmin_{\vy}\|\vy\|_{2,1}+(2\lambda)^{-1}\|\nabla\vx_5^{k}-\vy+\lambda\vz_5^k\|_2^2,\\
\vx_5^{k+1} = &\argmin_{\vx}{\alpha\over2}\|\vx-\vb\|_2^2+(2\lambda)^{-1}\|\nabla\vx-\vy_5^{k+1}+\lambda\vz_5^k\|_2^2,\\
\vz_5^{k+1}=&\vz_5^k+\lambda^{-1}(\nabla\vx_5^{k+1}-\vy_5^{k+1}).
\end{align}
\end{subequations}
\end{enumerate}

\begin{corollary}\label{cor:TV}
Let $\vx_5^0=\vb+\alpha^{-1}\textnormal{div }\vz_5^0$. If the initialization for all algorithms~\eqref{eqn:TV-P}-\eqref{eqn:TV-P2} satisfies $\vy_1^0 =  - \vz_3^0  =\nabla \vx_4^0-\lambda (\vv_4^0-\vv_4^{-1}) =\vy_5^{1}$ and $\vz_1^0 = \vv_3^0 =\vv_4^0 = \vz_5^0+\lambda^{-1}(\nabla \vx_5^0-\vy_5^{1})$. Then for $k\geq1$, we have the following equivalence results between the iterations of the four algorithms:
\begin{align*}
\begin{array}{llll}
\vy_1^k &=  - \vz_3^k  &=\nabla \vx_4^k-\lambda (\vv_4^k-\vv_4^{k-1}) &=\vy_5^{k+1},\\
\vz_1^k &= \vv_3^k &=\vv_4^k &= \vz_5^k+\lambda^{-1}(\nabla \vx_5^k-\vy_5^{k+1}).
\end{array}
\end{align*}
\end{corollary}

\begin{remark} In any of the four algorithms, the $\nabla$ or $\textnormal{div}$ operator is separated in a different subproblem from the term $\|\cdot\|_{2,1}$ or its dual norm $\|\cdot\|_{2,\infty}$. The $\nabla$ or $\textnormal{div }$ operator is translation invariant so their subproblems can be solved by a diagonalization trick~\cite{FTVd}. The subproblems involving the term $\|\cdot\|_{2,1}$ or the indicator function $\iota_{\{\vv:\|\vv\|_{2,\infty}\leq1\}}$ have closed-form solutions. Therefore, in addition to the equivalence results, all the four algorithms have essentially the same per-iteration costs.
\end{remark}

\section*{Acknowledgments}
This work is supported by NSF Grants DMS-1349855 and DMS-1317602 and ARO MURI Grant W911NF-09-1-0383. We thank Jonathan Eckstein for bringing his early work~\cite[Chapter~3.5]{eckstein1989splitting} and~\cite{eckstein1994some} to our attention.

\bibliographystyle{siam}
\bibliography{PDequal,admm}
\end{document}